\documentclass{article}

\usepackage{amssymb}
\usepackage{latexsym}
\usepackage{amsthm}
\usepackage{enumerate}
\usepackage{epsfig}
\usepackage{color}
\usepackage{float}
\usepackage{subfigure}
\usepackage{amsmath}
\usepackage{ifthen}
\usepackage{lscape}
\usepackage{setspace}
\usepackage{amscd}
\usepackage{xspace}
\usepackage[nobysame]{amsrefs}
\usepackage{srcltx}

\def\a{\alpha}

\def\b{\beta}
\def\C{\mathcal{C}}
\def\g{\gamma}

\def\D{\Delta}
\def\H{\mathcal{H}}
\def\k{\kappa}
\def\L{\Lambda}
\def\l{\lambda}
\def\N{\mathbb{N}}

\def\R{\mathbb{R}}
\def\s{\sigma}
\def\S{\Sigma}

\def\w{\omega}
\def\Z{\mathbb{Z}}
\def\d{\partial}
\def\cross{\times}

\def\U{\mathcal{U}}
\def\e{\epsilon}
\def\half{\tfrac{1}{2}}

\def\t{\tau}
\def\P{\mathbb{P}}
\def\le{\leqslant}
\def\ge{\geqslant}

\def\PMF{\mathcal{PMF}}

\def\Bhat{\widehat B}
\def\Gbar{\overline{G}}
\def\Grel{\widehat G}

\def\gz{G^{\Z_+}}
\def\fmin{\mathcal{F}_{min}}
\def\ds{\mathcal{D}}

\def\Ksplit{K_1}
\def\Kn{K_2}
\def\Knpp{K_{3}}

\def\Knpphalf{K_{5}}
\def\Knested{K_6}
\def\Kgp{K_{7}}
\def\Knbd{K_{8}}
\def\Kpa{K_{9}}
\def\Ktrans{K_{10}}
\def\Kexp{K_{11}}
\def\Kqc{K_{12}}

\def\fix{\text{fix}}

\def\pA{pseudo-Anosov\xspace}
\def\qg{quasigeodesic\xspace}
\def\qgs{quasigeodesics\xspace}
\def\npp{nearest point projection\xspace}
\def\npps{nearest point projections\xspace}

\newcommand{\norm}[1]{|#1|}
\newcommand{\dhat}[1]{\widehat d (#1)}
\newcommand{\nhat}[1]{\dhat{1,#1}}
\newcommand{\dc}[1]{d_\C(#1)}

\newcommand{\mun}[1]{\mu_{#1}}
\newcommand{\rmun}[1]{\widetilde \mu_{#1}}
\newcommand{\rmu}{\widetilde \mu}
\newcommand{\rnu}{\widetilde \nu}
\newcommand{\gp}[2]{(#1 | #2)}
\newcommand{\Hbar}[1]{\overline{H(#1)}}

\newcommand{\pref}[1]{Proposition \ref{prop:#1}\xspace}

\newtheorem{theorem}{Theorem}[section]
\newtheorem{lemma}[theorem]{Lemma}

\newtheorem{proposition}[theorem]{Proposition}

\newtheorem*{theorem:main}{Theorem \ref{theorem:main}}

\theoremstyle{definition}

\newtheorem{claim}[theorem]{Claim}

\newcounter{case}

\restylefloat{figure}

\begin{document}

%\doublespace

\title{Random Heegaard splittings}
\author{Joseph Maher\footnote{email: joseph.maher@csi.cuny.edu}}
\date{\today}

\maketitle

\begin{abstract}
  Consider a random walk on the mapping class group, and let $w_n$ be
  the location of the random walk at time $n$.  A random Heegaard
  splitting $M(w_n)$ is a $3$-manifold obtained by using $w_n$ as the
  gluing map between two handlebodies. We show that the joint
  distribution of $(w_n, w_n^{-1})$ is asymptotically independent, and
  converges to the product of the harmonic and reflected harmonic
  measures defined by the random walk. We use this to show that the
  translation length of $w_n$ acting on the curve complex, and the
  distance between the disc sets of $M(w_n)$ in the curve complex,
  grows linearly in $n$.  In particular, this implies that a random
  Heegaard splitting is hyperbolic with asymptotic probability one.

Subject code: 37E30, 20F65, 57M50.

\end{abstract}

\tableofcontents

%%%%%%%%%%%%%%%%%%%%%%%%%%%%%%%%%%%%%%%%%%%%%%%%%%%%%%%%%%%%%%%%%%%%%%%%%%%%%%%
\section{Introduction}
%%%%%%%%%%%%%%%%%%%%%%%%%%%%%%%%%%%%%%%%%%%%%%%%%%%%%%%%%%%%%%%%%%%%%%%%%%%%%%%

Let $\S$ be a closed orientable surface of genus at least two.  The
mapping class group $G$ of $\S$ is the group of orientation preserving
diffeomorphisms of $\S$, modulo those isotopic to the identity. A
random walk on $G$ is a Markov chain on $G$ with transition
probabilities $p(x,y) = \mu(x^{-1}y)$, where $\mu$ is a probability
distribution on $G$, and we will always assume that the starting point
at time zero is the identity element in the group.  The path space of
the random walk is the probability space $(\gz, \P)$, where $w \in
G^{\Z_+}$ is a sample path, namely a sequence of group elements $w_n$
corresponding to the location of the sample path at time $n$.  The
measure $\P$ restricted to the $n$-th factor gives the distribution of
$w_n$, which is the $n$-fold convolution of $\mu$ with itself, which
we shall denote by $\mun{n}$.  A \emph{random Heegaard splitting}
$M(w_n)$ is a $3$-manifold constructed by gluing two handlebodies
together using a random walk of length $n$ as the gluing map. The
\emph{distance} $d(M(w_n))$ of a Heegaard splitting is defined in
terms of the complex of curves, which is a simplicial complex whose
vertices are isotopy classes of simple closed curves in the surface,
and whose simplicies are spanned by collections of disjoint simple
closed curves. The disc set of a handlebody is the collection of
simple closed curves which bound discs in the handlebody. The distance
of a Heegaard splitting is the minimal distance between the disc sets
of the two handlebodies in the complex of curves.  In this paper, we
show that the distance of a random Heegaard splitting grows linearly
with the length of the random walk, assuming the following
restrictions on the support of $\mu$.  We say a subgroup of the
mapping class group is complete if the endpoints of its \pA elements
are dense in $\PMF$, and we require the semi-group generated by the
support of $\mu$ to contain a complete subgroup of the mapping class
group.

\begin{theorem} 
\label{theorem:main}
Let $G$ be the mapping class group of a closed orientable surface of
genus at least two.  Consider a random walk generated by a finitely
supported probability distribution $\mu$ on $G$, whose support
generates a semi-group containing a complete subgroup of the mapping
class group.  Then there are constants $\ell_2 \ge \ell_1 > 0$ such
that
\[ \P( \ell_1 \le \tfrac{1}{n}d(M(w_n)) \le \ell_2 ) \to 1 \text{ as }
n \to \infty, \] 
where $d(M(w_n))$ is the distance of the Heegaard splitting determined
by $w_n$.
\end{theorem}

In particular, the probability that the Heegaard splitting distance of
$M(w_n)$ is greater than two tends to one. If a manifold has Heegaard
splitting distance greater than two then Kobayashi \cite{kobayashi}
and Hempel \cite{hempel} show that the manifold is irreducible,
atoroidal and not Seifert fibered, which implies that the manifold is
hyperbolic, by geometrization, due to Perelman \cites{perelman1,
  perelman2}. This confirms a conjecture of W. Thurston from
\cite{thurston}, in which he suggests that most manifolds produced
from Heegaard splittings of a fixed genus should be hyperbolic. This
may also be thought of as a result similar in spirit to Thurston's
Dehn surgery theorem, which says that most fillings on a hyperbolic
manifold with toroidal boundary give rise to hyperbolic manifolds; we
have shown that most Heegaard splittings give rise to hyperbolic
manifolds, though of course, for a very different definition of most.
More recently, Dunfield and W. Thurston \cite{dt} have conjectured
that the volume of these random Heegaard splittings should grow
linearly with the length of the random walk. This follows from our
work, and some recent results announced by Brock and Souto
\cite{bs}. Lustig and Moriah \cite{lm} have shown that high distance
splittings are generic for an alternative definition of genericity,
using the Lebesgue measure class on $\PMF$.

The growth of the splitting distance has other implications for the
resulting $3$-manifolds.  By work of Hartshorn \cite{hartshorn},
Theorem \ref{theorem:main} shows that the probability that $M(w_n)$
contains an incompressible surface of genus at most $N$ tends to zero
as $n$ tends to infinity, and by Scharlemann and Tomova \cite{st}, the
probability that the manifold has a Heegaard splitting of lower genus
tends to zero as $n$ tends to infinity.

The Torelli group is the subgroup of the mapping class group which
acts trivially on homology. Removing a standard handlebody from $S^3$
and gluing it back in using an element of the Torelli group produces a
homology sphere. The Torelli group is complete, and so these results
apply to the nearest neighbour random walk on the Torelli group, which
may be used to produce random Heegaard splittings with trivial
homology, which are also hyperbolic with asymptotic probability one.

%%%%%%%%%%%%%%%%%%%%%%%%%%%%%%%%%%%%%%%%%%%%%%%%%%%%%%%%%%%%%%%%%%%%%%%%%%%%%%%
\subsection{Outline}
%%%%%%%%%%%%%%%%%%%%%%%%%%%%%%%%%%%%%%%%%%%%%%%%%%%%%%%%%%%%%%%%%%%%%%%%%%%%%%%

Masur and Minsky \cite{mm1} showed that the mapping class group is
weakly relatively hyperbolic. In particular, they showed that the
relative space is quasi-isometric to the complex of curves, which is
$\delta$-hyperbolic, though not proper. Klarreich \cite{klarreich}
identified the Gromov boundary of the complex of curves as $\fmin$,
the space of minimal foliations on the surface.  Using work of
Kaimanovich and Masur \cite{km} one may show that a random walk on the
mapping class group converges to a minimal foliation in the Gromov
boundary almost surely. This defines a measure $\nu$, known as
harmonic measure, on the boundary, where the measure of a subset is
just the probability that a sample path of the random walk converges
to a foliation contained in that subset. This measure depends on the
choice of probability distribution used to generate the random walk.
For example, if we consider the \emph{reflected} random walk generated
by $\rmu(g) = \mu(g^{-1})$, then if $\mu$ is not symmetric, this may
produce a different harmonic measure, which we shall call the
reflected harmonic measure, and denote $\rnu$.

Masur and Minsky \cite{mm3} showed that the disc set $\ds$ is a
quasiconvex subset of the complex of curves. Consider a \pA element
$g$ of the mapping class group, whose stable and unstable foliations
are disjoint from the limit set of the disc set. As the disc set is
quasiconvex, its nearest point projection onto an axis for $g$ has
bounded diameter, and so if the translation length of $g$ is large,
the nearest point projections of $\ds$ and $g\ds$ to the axis will be
disjoint. In fact, the distance between the two discs sets will be
roughly the translation length of $g$ minus the diameter of the
projection of the disc set to the axis.

We need to show this situation is generic for group elements arising
from random walks. In \cite{maher1} we showed that a random walk $w_n$
gives a \pA element with asymptotic probability one. However, we also
need to show that its stable and unstable foliations lies outside the
limit set of the disc set with asymptotic probability one. We start by
showing that the harmonic measure of the limit set of the disc set is
zero, using a result of Kerckhoff \cite{kerckhoff} that says any train
track may be split finitely many times to produce a train track
disjoint from the disc set.  We then show that the distribution of
pairs of stable and unstable foliations $(\l^+(w_n), \l^-(w_n))$ in
$\fmin \cross \fmin$ converges to the product of harmonic measure and
reflected harmonic measure on $\fmin \cross \fmin$. We briefly explain
why one would expect this to be the case.  The first observation is
that if $g$ is a \pA element of large translation length, then the
stable foliation of $g$ lies in the boundary of the halfspace
$H(1,g)$, where the halfspace $H(1,g)$ consists of all points in the
relative space at least as close to $g$ as to $1$. This is an
elementary property of hyperbolic translations in a
$\delta$-hyperbolic space, and the corresponding statement follows for
the unstable foliations, which is contained in the halfspace
$H(1,g^{-1})$. This indicates that the distribution of the pairs of
stable and unstable foliations should be closely related to the
distribution of pairs $(w_n, w_n^{-1})$. A sample path $w$ converges
to a minimal foliation $F(w)$ almost surely, and it can be shown that
for large $n$, the element $w_{2n}$ lies in the halfspace $H(1, w_n)$
determined by the midpoint $w_n$ of the sample path $w$, with
asymptotic probability one.  This is illustrated on the left hand side
of Figure \ref{picture6} below.

\begin{figure}[H] 
\begin{center}
\epsfig{file=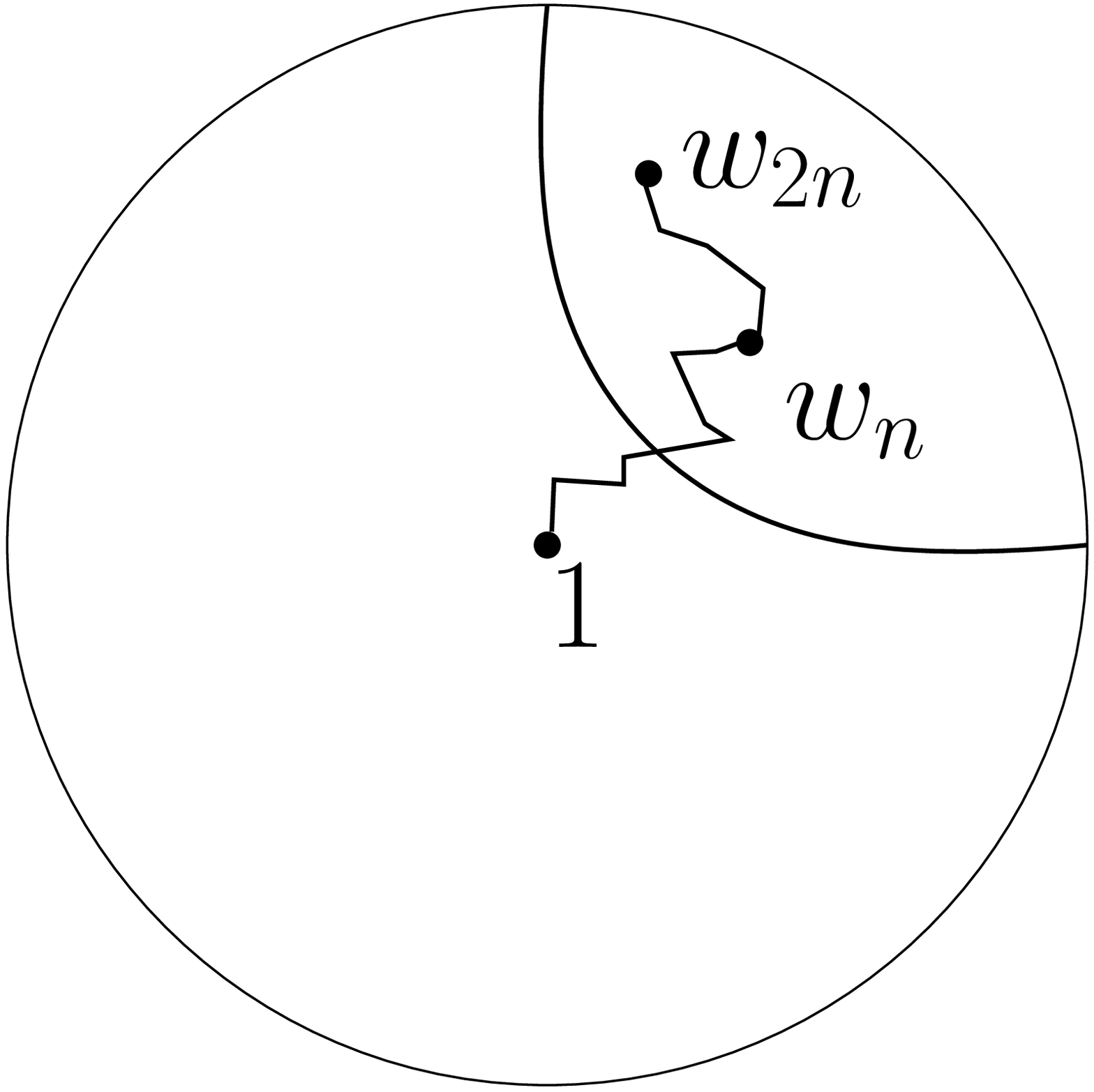, height=100pt}
\end{center}
\caption{Estimating the distribution of stable and unstable foliations.}\label{picture6}
\end{figure}

The corresponding statement therefore holds for the inverse of the
sample path, i.e. the endpoint of the inverse path $w_{2n}^{-1}$ lies
in the halfspace determined by the midpoint of the sample path.  If we
denote the increments of the sample path by $a_i$, then $w_{2n}^{-1}$
is equal to $a_{2n}^{-1} \ldots a_1^{-1}$, so the midpoint is given by
$a_{2n}^{-1} \ldots a_{n+1}^{-1}$. This is illustrated on the right
hand side of Figure \ref{picture6} above.  Therefore the distribution
of pairs $(w_{2n}, w_{2n}^{-1})$ is closely related to the
distribution of midpoints of sample paths $(w_n, w_{2n}^{-1} w_n)$.
However, the first half of the sample path $w_n = a_1 \ldots a_n$ is
independent of the second half of the sample path $a_{n+1} \ldots
a_{2n}$, and so independent of its inverse $w_{2n}^{-1}w_n$, as well.
Therefore these pairs are independently distributed according to the
product $\mun{n} \cross \rmun{n}$, which converges to $\nu \cross
\rnu$, namely the product of harmonic measure and reflected harmonic
measure on $\fmin \cross \fmin$. To make this argument precise, we use
the fact that as sample paths converge to the boundary, the relative
length of the sample path $\nhat{w_n}$ tends to infinity almost
surely.  Therefore we may estimate the probability that the pair of
stable and unstable foliations lie in a particular subset $A \cross B$
of $\fmin \cross \fmin$ in terms of the probability that the pair
$(w_n, w_n^{-1} )$ lies in a ``halfspace neighbourhood'' $N_r(A)
\cross N_r(B)$ of this set, where $N_r(X)$ is the union of all
halfspaces $H(1,x)$ with $\nhat{x} \ge r$ which intersect $X$. We will
be able to get upper and lower bounds for the different measures of $A
\cross B$ in terms of various $N_r$-neighbourhoods of the original
set, and the asymptotic results will then follow from taking limits as
$r$ tends to infinity, as $\bigcap N_r(X) = X$ for any closed subset
of $\fmin$.

The fact that the asymptotic distribution of stable and unstable
endpoints of \pA elements is $\nu \cross \rnu$ shows that the
probability that $w_n$ is \pA with neither endpoint in the disc set is
asymptotically one. Consider a pair of disjoint closed sets $A$ and
$B$ of $\fmin$, which are also both disjoint from the limit set of the
disc set. There is an upper bound on the distance of any geodesic with
one endpoint in $A$ and one endpoint in $B$ from the identity element
of the group. Furthermore, there is an upper bound on the diameter of
the nearest point projection of the disc set $\ds$ onto any such
geodesic.  The proportion of elements of a random walk of length $n$
which are \pA with stable foliation in $A$ and unstable foliation in
$B$ is asymptotic to $\nu(A)\rnu(B)$. As the relative length of $w_n$
grows linearly in $n$, as shown in \cite{maher2}, this shows that the
distance between $\ds$ and $w_n \ds$ also grows linearly for this
collection of elements.  However, the complement of the collection of
such sets $A \cross B$ in $\fmin \cross \fmin$ is just the union of
the diagonal with $\ds \cross \fmin$ and $\fmin \cross \ds$, and so
has measure zero. Therefore, the probability that the splitting
distance is roughly linear in $n$ tends to one asymptotically.

We now give a brief overview of the organization of the paper. In
Section \ref{section:preliminaries} we review some standard
definitions and previous results.  In section \ref{section:disc set},
we show that the limit set of the disc set has harmonic measure zero,
and then in Section \ref{section:independence}, we show that the joint
distribution of stable and unstable endpoints of \pA elements
converges to the product of harmonic and reflected harmonic measure on
$\fmin \cross \fmin$. Finally, in Section \ref{section:distance}, we
show that the distance of a random Heegaard splitting grows linearly
in the length of the random walk.

%%%%%%%%%%%%%%%%%%%%%%%%%%%%%%%%%%%%%%%%%%%%%%%%%%%%%%%%%%%%%%%%%%%%%%%%%%%%%%%
\subsection{Acknowledgements} 
%%%%%%%%%%%%%%%%%%%%%%%%%%%%%%%%%%%%%%%%%%%%%%%%%%%%%%%%%%%%%%%%%%%%%%%%%%%%%%%

I would like to thank the referee for many helpful comments and
suggestions. I am extremely grateful to the referee of \cite{maher1}
for pointing out a gap in an earlier version of these results.  This
work was partially supported by NSF grant DMS-0706764.

%%%%%%%%%%%%%%%%%%%%%%%%%%%%%%%%%%%%%%%%%%%%%%%%%%%%%%%%%%%%%%%%%%%%%%%%%%%%%%%
\section{Preliminaries} \label{section:preliminaries}
%%%%%%%%%%%%%%%%%%%%%%%%%%%%%%%%%%%%%%%%%%%%%%%%%%%%%%%%%%%%%%%%%%%%%%%%%%%%%%%

In this section we review some standard definitions and previous
results.

%%%%%%%%%%%%%%%%%%%%%%%%%%%%%%%%%%%%%%%%%%%%%%%%%%%%%%%%%%%%%%%%%%%%%%%%%%%%%%%
\subsection{The mapping class group and the complex of curves}
%%%%%%%%%%%%%%%%%%%%%%%%%%%%%%%%%%%%%%%%%%%%%%%%%%%%%%%%%%%%%%%%%%%%%%%%%%%%%%%

Let $\S$ be an orientable surface of finite type, i.e. a surface of
genus $g$ with $p$ marked points, usually referred to as punctures.
The mapping class group $G$ of $\S$ consists of orientation preserving
diffeomorphisms of $\S$ which preserve the punctures, modulo those isotopic to
the identity. For the purposes of this paper we shall assume that $\S$
is not a sphere with three or fewer punctures.

The collection of essential simple closed curves in the surface may be
made in to a simplicial complex, called the the \emph{complex of
  curves}, which we shall denote $\C(\S)$.  The vertices of this
complex are isotopy classes of simple closed curves in $\S$, and a
collection of vertices spans a simplex if representatives of the
curves can be realised disjointly in the surface.  The complex of
curves is a finite dimensional simplicial complex, but it is not
locally finite. We will write $\dc{x,y}$ for the distance in the
one-skeleton between two vertices $x$ and $y$ of the complex of
curves. We will fix a basepoint for the complex of curves, which we
shall denote $x_0$. The mapping class group acts by simplicial
isometries on the complex of curves. For certain sporadic surfaces the
definition above produces a collection of disconnected points, and so
a slightly different definition is used. If the surface is a torus
with at most one puncture, then two vertices are connected by an edge
if the corresponding simple closed curves may be isotoped to intersect
transversely exactly once. If the surfaces is a four punctured sphere,
then two vertices are connected by an edge if the corresponding simple
closed curves may be isotoped to intersect transversely in two points.
In both of these cases, the resulting curve complex is isomorphic to
the Farey graph.

A geodesic metric space is \emph{$\delta$-hyperbolic} if every
geodesic triangle is $\delta$-slim, i.e. each edge is contained in a
$\delta$-neighbourhood of the other two.  Masur and Minsky \cite{mm1}
have shown that the complex of curves is $\delta$-hyperbolic.

The mapping class group is finitely generated, so any choice of
generating set $A$ gives rise to a word metric on $G$, and any two
different choices of finite generating set give quasi-isometric word
metrics.  Given a group $G$, and a collection of subgroups $\H = \{
H_i \}_{i \in I }$, we define the \emph{relative length} of a group
element $g$ to be the length of the shortest word in the typically
infinite generating set $A \cup \H$.  This defines a metric on $G$
called the \emph{relative metric}, which depends on the choice of
subgroups $\H$.  We will write $\Grel$ to denote the group $G$ with
the relative metric, which we shall also refer to as the
\emph{relative space}. We will write $\dhat{a,b}$ for the relative
distance between two group elements $a$ and $b$, and we will also
write $\dhat{A,B}$ for the smallest distance between two sets $A$ and
$B$. We say a finitely generated group $G$ is \emph{weakly relatively
  hyperbolic}, relative to a finite list of subgroups $\H$, if the
relative space $\Grel$ is $\delta$-hyperbolic.

We may consider the relative metric on the mapping class group with
respect to the following collection of subgroups.  Let $\{\alpha_1,
\ldots, \alpha_n\}$ be a list of representatives of orbits of simple
closed curves in $\S$, under the action of the mapping class group.
Let $H_i = \fix{(\alpha_i)}$ be the subgroup of $G$ fixing $\alpha_i$.
Masur and Minsky \cite{mm1} have shown that the resulting relative
space is quasi-isometric to the complex of curves. The quasi-isometry
is given by $g \mapsto g(x_0)$, where $x_0$ is a fixed basepoint for
the complex of curves.  As the complex of curves is
$\delta$-hyperbolic, this shows that the mapping class group is weakly
relatively hyperbolic. The Gromov boundary of the complex of curves,
and hence the relative space, is the set formed by taking the
foliations in $\PMF$, the space of projective measured foliations on
the surface, whose trajectories contain no closed curves, and
forgetting the measures, as shown by Klarreich \cite{klarreich}, see
also Hamensd\"adt \cite{hamenstadt}.  This space, denoted $\fmin$, is
Hausdorff, but not compact. We will write $\Gbar$ for the union of
$\Grel$ with its Gromov boundary $\fmin$.

A \pA element $g$ of the mapping class group acts as a hyperbolic
isometry on the relative space, and so has a unique pair of fixed
points in the boundary, which we shall call the stable and unstable
foliations of $g$.  Following Ivanov \cite{ivanov}, we say that a
subgroup $H$ of the mapping class group is \emph{complete}, if the
stable and unstable foliations of \pA elements of $H$ are dense in
$\PMF$.

%%%%%%%%%%%%%%%%%%%%%%%%%%%%%%%%%%%%%%%%%%%%%%%%%%%%%%%%%%%%%%%%%%%%%%%%%%%%%%%
\subsection{Random walks}
%%%%%%%%%%%%%%%%%%%%%%%%%%%%%%%%%%%%%%%%%%%%%%%%%%%%%%%%%%%%%%%%%%%%%%%%%%%%%%%

We now review some background on random walks on groups, see for
example Woess \cite{woess}.  Let $G$ be the mapping class group of an
orientable surface of finite type, which is not a sphere with three or
fewer punctures, and let $\mu$ be a probability distribution on $G$.
We may use the probability distribution $\mu$ to generate a Markov
chain, or \emph{random walk} on $G$, with transition probabilities
$p(x,y) = \mu(x^{-1}y)$. We shall always assume that we start at time
zero at the identity element of the group.  The \emph{path space} for
the random walk is the probability space $(G^{\Z_+},\P)$, where
$G^{\Z_+}$ is the set of all infinite sequences of elements $G$. We
will write $w$ for an element of the path space $G^{\Z_+}$, which we
shall call a \emph{sample path}, and we will write $w_n$ for the
$n$-th coordinate of $w$, so $w_n$ is the position of the sample path
at time $n$.  The position of the random walk at time $n$ may be
described as the product $a_1 a_2 \dots a_n$, where the $a_i$ are the
\emph{increments} of the random walk, i.e. the $a_i$ are a sequence of
independent $\mu$-distributed random variables. Therefore the
distribution of random walks at time $n$ is given by the $n$-fold
convolution of $\mu$, which we shall write as $\mun{n}$, and we shall
write $p^{(n)}(x,y)$ for the probability that you go from $x$ to $y$
in $n$ steps. 
% The probability measure $\P$ is determined by $\mun{n}$ using the
% Kolmogorov extension theorem.

We shall always require that the group generated by the support of
$\mu$ is \emph{non-elementary}, which means that it contains a pair of
\pA elements with distinct fixed points in $\PMF$. We do not assume
that the probability distribution $\mu$ is symmetric, so the group
generated by the support of $\mu$ may be strictly larger than the
semi-group generated by the support of $\mu$. Throughout this paper we
will need to assume that the probability distribution $\mu$ has finite
support, and in order to show that the disc set has harmonic measure
zero, we will need to assume that the semi-group generated by the
support of $\mu$ contains a complete subgroup of the mapping class
group.

In \cite{maher1}, we showed that it followed from results of
Kaimanovich and Masur \cite{km} and Klarreich \cite{klarreich}, that a
sample path converges almost surely to a uniquely ergodic, and hence
minimal, foliation in the Gromov boundary of the relative space. This
gives a measure $\nu$ on $\fmin$, known as \emph{harmonic measure}.
The harmonic measure $\nu$ is $\mu$-stationary, i.e. \[ \nu(X) =
\sum_{g \in G} \mu(g)\nu(g^{-1}X). \]

\begin{theorem} \cites{km, klarreich, maher1}
\label{theorem:converge}
Consider a random walk on the mapping class group of an orientable
surface of finite type, which is not a sphere with three or fewer
punctures, determined by a probability distribution $\mu$ such that
the group generated by the support of $\mu$ is non-elementary. Then
a sample path $\{ w_n \}$ converges to a uniquely ergodic
foliation in the Gromov boundary $\fmin$ of the relative space $\Grel$
almost surely, and the distribution of limit points on the boundary
is given by a unique $\mu$-stationary measure $\nu$ on $\fmin$.
\end{theorem}

We remark that the measure $\nu$ is supported on the uniquely ergodic
foliations, which are a subset of $\PMF$, so it is possible to
consider $\nu$ as a measure on $\PMF$, with zero weight on all the
non-uniquely ergodic measures.

It will also be convenient to consider the \emph{reflected} random
walk, which is the random walk generated by the reflected measure
$\rmu$, where $\rmu(g) = \mu(g^{-1})$. We will write $\rnu$ for the
corresponding $\rmu$-stationary harmonic measure on $\fmin$. Note that
if $\mu$ has finite support, then so does $\rmu$, and if the
semi-group generated by the support of $\mu$ contains a complete
subgroup, then so does the semi-group generated by the support of
$\rmu$. We will write $H^+$ for the semi-group generated by the
support of $\mu$, and $H^-$ for the semi-group generated by the
support of $\rmu$.

%%%%%%%%%%%%%%%%%%%%%%%%%%%%%%%%%%%%%%%%%%%%%%%%%%%%%%%%%%%%%%%%%%%%%%%%%%%%%%%
\subsection{Coarse geometry}
%%%%%%%%%%%%%%%%%%%%%%%%%%%%%%%%%%%%%%%%%%%%%%%%%%%%%%%%%%%%%%%%%%%%%%%%%%%%%%%

In this section we review some useful facts about $\delta$-hyperbolic
spaces, see for example, Bridson and Haefliger \cite{bh}. Recall that
a geodesic metric space $X$ is $\delta$-hyperbolic, if for every
geodesic triangle with sides $\a, \b$ and $\g$, any side is contained
in a $\delta$-neighbourhood of the other two. We emphasize that we do
not require our metric space to be proper. We shall write $\dhat{a,b}$
for the distance between points $a$ and $b$ in $X$, and if $A$ and $B$
are subsets of $X$ we shall write $\dhat{A, B}$ for the infimum of the
distance between any pair of points $a \in A$ and $b \in
B$. Furthermore, we shall assume that our space $X$ has a basepoint,
which we shall denote $1$.

A $\delta$-hyperbolic space has a canonical boundary at infinity,
known as the Gromov boundary, which may be defined in terms of either
\qg rays, or the Gromov product. The Gromov boundary $\d X$ may be
defined as equivalence classes of \qg rays, where two rays are
equivalent if they stay a bounded Hausdorff distance apart.  Let $I$
be a connected subset of $\R$. A \emph{$(K,c)$-\qg} is a map $\g:I \to
\Grel$ such that
\[ \tfrac{1}{K} \dhat{\g_s,\g_t} -c \le \norm{s-t} \le K
\dhat{\g_s,\g_t} +c. \]
We shall write $\g_t$ for $\g(t)$, and if the domain $I$ of $\g$ is a
half-infinite interval, we shall say $\g$ is a \qg ray.  For the
purposes of this paper, our $\delta$-hyperbolic metric spaces will
always be either $\Grel$ or the complex of curves $\C(\S)$, and we can
always restrict attention to taking distances between group elements
in the first case, or simple closed curves in the second case. This
means that our spaces are effectively discrete spaces, so we may
replace the domain $\R$ with $\Z$, and the distance function takes
values in $\Z$. In this case, a $(K,c)$-\qg is a $(K',0)$-\qg, where
$K'$ may be chosen to be $Kc + K + c$.  Therefore we can always choose
$c$ to be zero, so we will just write $K$-\qg for a $(K,0)$-\qg.

In a $\delta$-hyperbolic space, \qgs are characterised by the property
that they lie in a bounded neighbourhood of a geodesic, and their
nearest point projections to the geodesic make coarsely linear
progress along the geodesic.

\begin{proposition} \cite{bh}*{Theorem 1.7, III.H} 
\label{prop:qg}
Let $\a$ be a $K$-quasigeodesic in a $\delta$-hyperbolic space, and
suppose there is a geodesic $\b$ connecting its endpoints. Then the
Hausdorff distance between $\a$ and $\b$ is at most $L$, for some
constant $L$ which only depends on $K$ and $\delta$.
\end{proposition}

We will refer to the constant $L$ as the \emph{geodesic neighbourhood}
constant for the \qg. In an arbitrary $\delta$-hyperbolic space there
may be pairs of points in the Gromov boundary which are not connected
by geodesics. In the complex of curves, Masur and Minsky \cite{mm2}
show that there are special geodesics, which they call tight
geodesics, connecting any pair of points in the Gromov boundary. This
implies that in the relative space for the mapping class group,
$\Grel$, any two points in $\Grel \cup \d \Grel$ are connected by a
$\Ksplit$-\qg, for some constant $\Ksplit$ which works for all pairs
of points in the Gromov boundary. Such quasigeodesics may also be
constructed using train tracks, as discussed in Section
\ref{section:train tracks}.  Therefore using \pref{qg} with $K =
\Ksplit$, we will fix $\Kn$ to be a geodesic neighbourhood constant
for any $\Ksplit$-\qg.

The Gromov boundary may also be defined in terms of the Gromov
product.  The Gromov product of two points is $\gp{a}{b} =
\tfrac{1}{2}(\dhat{1,a} + \dhat{1,b} - \dhat{a,b})$, which in a
$\delta$-hyperbolic space is within $\delta$ of the distance from $1$
to a geodesic $[a,b]$. We say a sequence of points $a_n$ converges to
the boundary if $\gp{a_m}{a_n}$ tends to infinity as $m$ and $n$ tend
to infinity, and a sequence $a_n$ converges to the same limit point as
a convergent sequence $b_n$ if $\gp{a_n}{b_n}$ tends to infinity as
$n$ tends to infinity.  The Gromov product may be extended to points
in the boundary, i.e.  if $\a$ and $\b$ are points in the boundary,
then define 
\[ \gp{\a}{\b} = \sup \liminf_{m,n \to \infty} \gp{a_m}{b_n}, \] 
where the supremum is taken over all sequences $a_n \to \a$, and $b_n
\to \b$.

It is a standard property of $\delta$-hyperbolic spaces that the map
given by nearest point projection onto a geodesic is coarsely well
defined.  If two points $a$ and $b$ have nearest point projections
onto a geodesic $\g$ which are sufficiently far apart, then roughly
speaking, the fastest way to get from $a$ to $b$ is to head to the
closest point $p$ on the geodesic $\g$, then follow the geodesic to
the closest point $q$ to $b$, and then head back out to $b$. We will
call this path $[a,p] \cup [p,q] \cup [q,b]$ a \emph{\npp path} for
$a$ and $b$, relative to the geodesic $\g$.
%The central subinterval of the
%nearest point projection path contained in $\g$ is equal to $\g \cap
%\pi_\g(a,b)$. 
If either $a$ or $b$ lie on $\g$, then the first or last segment of
the nearest point projection path will have zero length. We summarize
some useful properties of nearest point projection paths in the next
proposition, see for example Coornaert, Delzant and Papadopoulos
\cite{cdp}.

\begin{proposition} 
{\bf (nearest point projection paths are almost geodesic)}
\cite{maher2}*{Proposition 3.4}
\label{prop:double}
Let $a$ and $b$ have \npps $p$ and $q$ respectively to a geodesic. If $\dhat{p,q} >
14\delta$ then
\begin{enumerate}
\item $\dhat{a,b} \geqslant \dhat{a,p} + \dhat{p,q} + \dhat{q,b} -
  24\delta$ \label{eq:npp1}
\item A nearest point projection path $[a, p] \cup [p, q] \cup [q, b]$
  is contained in a $6\delta$-neighbourhood of a geodesic $[a,b]$.
\item A nearest point projection path $[a, p] \cup [p, q] \cup [q, b]$
  is a $(1, 24\delta)$-\qg, and hence a $\Knpp$-\qg in a discrete
  space, for a constant $\Knpp$ which only depends on $\delta$.
%\item The distance from $x$ to the nearest point projection path
%  $\pi_\g(a,b)$ is at least $\dhat{x,[p,q]} - 48 \delta$.
\end{enumerate}
\end{proposition}

\begin{proof}
Parts \eqref{eq:npp1} and $(2)$ are taken directly from \cite{maher2}*{Proposition 3.4}, and
part $(3)$ is a direct consequence of part $(1)$. 
% Item $4$ follows from $1$, as if $p$ is closer to $x$ than $q$, then
% $[x,p] \cup [p,a]$ is a nearest point projection path between $x$
% and $a$.
\end{proof}

The proposition above gives a quantitative version of the fact that
nearest point projections are coarsely distance decreasing, i.e.
$\dhat{p,q} \leqslant \dhat{a,b} + 24\delta$, using the notation
above.

It will be useful to know that nearest point projections to
close paths are close. Given a geodesic $[x,y]$, and a path $\g$ from
$x$ to $y$ which stays in a bounded neighbourhood of the geodesic,
then the nearest point projection of any point $z$ to $[x,y]$ is a
bounded distance from the nearest point projection of $z$ to $\g$.

\begin{proposition} 
{\bf (nearest point projections to close paths are close)}
\cite{maher2}*{Proposition 3.6}
\label{prop:close}
Let $[x,y]$ be a geodesic, and let $\g$ be any path from $x$ to $y$
contained in a $K$-neighbourhood of $[x,y]$. Then for any point $z$,
the closest point to $z$ on $[x,y]$ is distance at most $3K + 6\delta$
from the closest point on $\g$ to $z$. 
\end{proposition}

We now show that nearest point projections to geodesics may be defined
for points in the Gromov boundary.  Let $[x,y]$ be a geodesic, and let
$D$ be the diameter of $1 \cup [x,y]$. Let $\l$ be a point in the
Gromov boundary, and let $\{ a_n \}$ be a sequence which converges to
$\l$. We define a \npp of $\l$ to $[x,y]$ to be a \npp of $a_m$ to
$[x,y]$, where $a_m$ is an element of the sequence $\{ a_n \}$ with
$\gp{a_m}{a_n} \ge D + 8 \delta$ for all $n \ge m$. We now show that
this \npp is coarsely well-defined for points in the boundary.

\begin{proposition} 
{\bf (\npp is coarsely well-defined on $X \cup \d X$)}
\label{prop:npp boundary}
Let $[x,y]$ be a geodesic, and let $\l \in \d X$. If $p$ and $q$ are
\npps of $\l$ to $[x,y]$, then $\dhat{p,q} \le 15
\delta$. Furthermore, for any sequence $\{ a_n \}$ which converges to
$\l$, there is an $N$ such that if $p$ is a \npp of $\l$ to $[x, y]$,
and $p_n$ is a \npp of $a_n$ to $[x,y]$, then $\dhat{p,p_n} \le 15
\delta$, for all $n \ge N$.
\end{proposition}

\begin{proof}
Let $D$ be the diameter of $1 \cup [x,y]$, and let $\{ a_n \}$ be a
sequence which converges to $\l$. As $a_n \to \l$, the Gromov
product $\gp{a_m}{a_n} \to \infty$, so there is a number $N$ such that
$\gp{a_m}{a_n} \ge D + 8 \delta$ for all $m,n \ge N$. Let $p_n$ be a
\npp of $a_n$ to $[x,y]$. If $\dhat{p_m, p_n} \ge 15 \delta$, then the
\npp path $[a_m, p_m] \cup [p_m, p_n] \cup [p_n, a_n]$ lies in a $6
\delta$-neighbourhood of a geodesic $[a_m, a_n]$, by \pref{double}, so
in particular $\dhat{[x,y], [a_m, a_n]} \le 6 \delta$. Therefore the
Gromov product $\gp{a_m}{a_n} \le D + 7 \delta$, a contradiction.
This shows that $\dhat{p_m, p_n} \le 15 \delta$. If $\{ a_n \}$ and
$\{ b_n \}$ are two sequences which converge to $\l$, then the
sequence which alternates between the terms of $a_n$ and $b_n$ also
converges to $\l$. Applying the argument above to the alternating
sequence implies that $\dhat{p,q} \le 15 \delta$ for any \npps of $\l$
to $[x,y]$, as required. The final statement follows by choosing a
constant $N$ as above for the alternating sequence.
\end{proof}

Finally, as an analogue of \pref{double}, we show that subsegments of
a bi-infinite \qg are contained in a bounded neighbourhood of a \npp
path.

\begin{proposition} 
{\bf (\npp paths for \qgs)}
\label{prop:npp qgs}
Let $[x,y]$ be a geodesic, let $\g$ be a $K$-\qg with endpoints $\l_1$
and $\l_2$ in $\d X$, and let $p_i$ be a \npp of $\l_i$ to $[x,y]$.
Then there are constants $K'$ and $L$, where $K'$ only depends on
$\delta$, and $L$ only depends on $K$ and $\delta$, such that if
$\dhat{p_1, p_2} \ge K'$, then for all $n$ sufficiently large, the
segment of $\g$ between $\g_n$ and $\g_{-n}$ is contained in an
$L$-neighbourhood of the \npp path $[\g_n, p_1] \cup [p_1, p_2] \cup
[p_2, \g_{-n}]$. In particular, $\dhat{\g_n, \g_{-n}} \ge \dhat{p_1,
  p_2} - 4L$.
\end{proposition}

\begin{proof}
The sequence $\g_n \to \l_1$ as $n \to \infty$, so by
\pref{npp boundary}, there is an $N_1$ such that $\dhat{p_1, q_n} \le
15 \delta$. Similarly, as $\g_{-n} \to \l_2$ as $n \to \infty$, there is
an $N_2$, such that $\dhat{p_2, q_{-n}} \le 15 \delta$, for all $n \ge
N_2$. We shall consider segments of $\g$ between $\g_n$ and $\g_{-n}$
with $n \ge N_1 + N_2$, and we shall choose $K' = 45 \delta$.

By \pref{qg}, there is an $L'$, which only depends on $K$ and
$\delta$, such that the segment of $\g$ between $\g_n$ and $\g_{-n}$
is contained in an $L'$-neighbourhood of a geodesic $[\g_n, \g_{-n}]$.
As $\dhat{p_1, q_n} \le 15 \delta$ and $\dhat{p_2, q_{-n}} \le 15
\delta$, this implies $\dhat{q_n, q_{-n}} \ge 15 \delta$, so by
\pref{double}, $[\g_n, \g_{-n}]$ is contained in a $6
\delta$-neighbourhood of $[\g_n, q_n] \cup [q_n, q_{-n}] \cup [q_{-n},
\g_{-n}]$.  By thin triangles, the segment of $\g$ between $\g_n$ and
$\g_{-n}$ is contained in an $L' + 22 \delta$-neighbourhood of the
\npp path $[\g_n, p_1] \cup [p_1, p_2] \cup [p_2, \g_{-n}]$.
Therefore we may choose $L = L' + 22 \delta$, which only depends on
$K$ and $\delta$, as required.
\end{proof}

%%%%%%%%%%%%%%%%%%%%%%%%%%%%%%%%%%%%%%%%%%%%%%%%%%%%%%%%%%%%%%%%%%%%%%%%%%%%%%%
\subsection{Train tracks} \label{section:train tracks}
%%%%%%%%%%%%%%%%%%%%%%%%%%%%%%%%%%%%%%%%%%%%%%%%%%%%%%%%%%%%%%%%%%%%%%%%%%%%%%%

We review some standard definitions concerning train tracks, as
described in Penner and Harer \cite{ph}. A \emph{train track} $\t$ on
a surface $S$ is an embedded $1$-complex, such that each edge, called
a \emph{branch}, is a smoothly immersed path with well-defined tangent
vectors at the endpoints, and at each vertex, called a \emph{switch},
the incident edges are mutually tangent. Every vertex must have edges
leaving in both possible directions. We require that the double of
every complementary region, with cusps deleted, has negative Euler
characteristic. In particular this means that the complementary
regions may not be monogons or bigons.

We say $\t$ is \emph{large} if the complementary regions of $\t$ are
polygons. We say $\t$ is \emph{maximal} if the complementary regions
of $\t$ are triangles. A train track is \emph{generic} if every switch
is trivalent. A trivalent switch has two branches leaving in one
direction, and a single branch leaving in the opposite direction. We
will call the pair of branches the \emph{incoming} branches, and the
single branch in the opposite direction the \emph{outgoing} branch.  A
\emph{train route} is a smooth path in $\t$. We say $\t$ is
\emph{recurrent} if every branch is contained in a closed train route.
A train track $\t$ is \emph{transversely recurrent} if every branch of
$\t$ is crossed by a simple closed curve $\a$ which intersects $\t$
transversely and \emph{efficiently}, i.e. $\a \cup \t$ has no bigon
components.  If $\t$ is both recurrent and transversely recurrent,
then we say that $\t$ is \emph{birecurrent}. We write $\s \prec \t$ if
$\s$ is a train track \emph{carried by} $\t$, which means that there
is a map homotopic to the identity on $\S$ such that every train route
on $\s$ is taken to a train route on $\t$.

An assignment of non-negative real numbers to each branch of the train
track determines a transverse measure supported on the train track, as
long as the switch condition is satisfied at each switch. The switch
condition says that the sum of weights on the incoming branches at a
switch is equal to the sum of weights on the outgoing branches.  Let
$N(\t) \subset \PMF(\S)$ be the polyhedron of projective classes of
measures supported on $\t$. The vertices of this polyhedron correspond
to simple closed curves carried by $\t$, which are called \emph{vertex
  cycles}. As there are only finitely many combinatorial classes of
train track for a given surface, the diameter of the vertex cycles is
bounded, and we define the distance between two train tracks,
$\dc{\s,\t}$ to be the distance between their vertex cycles in the
complex of curves. If the train track $\t$ is maximal and recurrent,
then $N(\t)$ has the same dimension as $\PMF(\S)$, and so the interior
is an open set.

Given a branch in a train track which meets the switch at each of its
endpoints as an outgoing branch, we can perform a move on this local
configuration, as shown in Figure \ref{picture11} below. This move is
called a \emph{split}.

\begin{figure}[H] 
\begin{center}
\epsfig{file=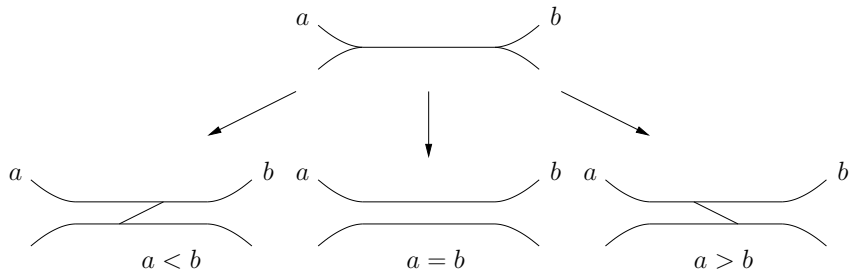, height=100pt}
\end{center}
\caption{Splitting a train track.}\label{picture11}
\end{figure}

A split gives rise to a new train track $\t'$ which is carried by the
original train track $\t$. The polyhedron of projective measures
$N(\t')$ is obtained by splitting the polyhedron $N(\t)$ along the
hyperplane $a=b$, in the notation of Figure \ref{picture11}. Given a
measured foliation $\l$ carried by a train track $\t$, a splitting
sequence for $\l$ is a sequence of train tracks $\t = \t_1 \prec \t_2
\prec \cdots$ obtained successive splits, such that each $\t_i$
carries $\l$.

A sequence of train tracks gives rise to a sequence in the complex of
curves, consisting of the vertex cycles of the train tracks, and hence
a sequence in the relative space.  Masur and Minsky \cite{mm1} show
that these splitting sequences are contained in uniformly bounded
neighbourhoods of geodesics in the relative space.  Splitting
sequences need not make uniform progress in the relative space, but
one may always pass to a subsequence which does. In particular, this
gives an explicit construction of $\Ksplit$-quasigeodesics between any
pair of points in $\Gbar$, for some constant $\Ksplit$, which only
depends on the surface.

\begin{theorem} 
\cite{mm1}
\label{theorem:uniform}
Let $\Gbar$ be the union of the relative space of the mapping class
group with its Gromov boundary. Then there is a constant $\Ksplit$,
which depends only on the surface, such that any pair of points in $\Gbar$ are
connected by a $\Ksplit$-\qg.
\end{theorem}

%%%%%%%%%%%%%%%%%%%%%%%%%%%%%%%%%%%%%%%%%%%%%%%%%%%%%%%%%%%%%%%%%%%%%%%%%%%%%%%
\subsection{Halfspaces}
%%%%%%%%%%%%%%%%%%%%%%%%%%%%%%%%%%%%%%%%%%%%%%%%%%%%%%%%%%%%%%%%%%%%%%%%%%%%%%%

A \emph{halfspace} $H(x,y)$ in a metric space consists of all those
points at least as close to $y$ as to $x$. We will be most concerned
with halfspaces $H(1,x)$, where the initial point is the identity
element in the mapping class group, so we will state some facts in
terms of halfspaces $H(1,x)$, even if they are true for more general
halfspaces.  Halfspaces do not necessarily behave well under
quasi-isometry, so we will always choose to work in the relative
space, and only apply results to the complex of curves if they are
independent of halfspaces.

We now review some useful properties of halfspaces in
$\delta$-hyperbolic spaces. In a $\delta$-hyperbolic metric space the
nearest point projection of a halfspace $H(x,y)$ to a geodesic $[x,y]$
connecting $x$ and $y$ lies in a bounded neighbourhood of the
halfspace.

\begin{proposition} 
\cite{maher2}*{Proposition 3.7}
\label{prop:half}
The nearest point projection of $H(x,y)$ to a geodesic $[x,y]$ is
contained in a $3\delta$-neighbourhood of $H(x,y)$.
\end{proposition}

The contrapositive of this says that if the nearest point projection
of a point onto $[x,y]$ lies outside a $3\delta$-neighbourhood of
$H(x,y)$, then the point lies in the halfspace $H(y,x)$. We now show
an analogous result for the closure of a halfspace.

\begin{proposition} 
{\bf (nearest point projections of halfspaces)}
\label{prop:npphalf}
There is a constant $\Knpphalf$, which only depends on $\delta$, such
that the nearest point projection of $\overline{H(1,x)}$ to $[1,x]$ is
contained in a $\Knpphalf$-neighbourhood of $H(1,x)$.
\end{proposition}

\begin{proof}
Let $a_n$ be a sequence of points in $H(1,x)$ which converge to a
point $\l$ in the boundary. Let $p$ be a nearest point projection of
$\l$ to $[x,y]$, and let $p_n$ be nearest point projections of $a_n$
to $[x,y]$. By \pref{npp boundary}, there is an $N$
such that the $p_n$ all lie within $15 \delta$ of $p$ for all $n \ge
N$. As the $p_n$ lie in a
$3 \delta$-neighbourhood of $H(1,x)$, by \pref{half},
this implies that the nearest point projection $p$ of any point in the
closure of $H(1,x)$ lies in a $18 \delta$-neighbourhood of $H(1,x)$.
Therefore we may choose $\Knpphalf$ to be $18 \delta$, which only
depends on $\delta$, as required.
\end{proof}

Let $H(1, x)$ be a halfspace, and let $y$ be a point on $[1, x]$ with
$\dhat{1, y} = \dhat{1, x} - K$. Then we say that $H(1, x)$ is
\emph{$K$-nested} with respect to $H(1, y)$. We now show that there is
a $K$, which only depends on $\delta$, such that if $H(1, x)$ is
$K$-nested inside $H(1, y)$, then $H(1, x) \subset H(1,y)$, and
furthermore, the closure of $H(1,x)$ is disjoint from the closure of
$H(y,1)$.

\begin{proposition}
{\bf (nested halfspaces)}
\label{prop:nested}
There is a constant $\Knested$, which only depend on $\delta$, such
that if $y$ lies on a geodesic $[1,x]$ with $\nhat{x} = \nhat{y} +
\Knested$, then $\Hbar{1, x}$ and $\Hbar{y, 1}$ are disjoint. In
particular $\Hbar{1, x} \subset \Hbar{1, y}$.
\end{proposition}

\begin{proof}
We shall choose $\Knested = 3 \Knpphalf + \delta$. By \pref{npphalf}
above, the \npps of $\Hbar{1, x}$ and $\Hbar{y, 1}$ are distance at
least $\Knpphalf + \delta$ apart, and so are disjoint. In particular,
$H(1, x) \subset H(1, y)$.
\end{proof}

We now find a lower bound for the Gromov product of any two points
contained in the closure of a halfspace $H(1,x)$.

\begin{proposition} 
{\bf (Gromov product bounded below in a halfspace)}
\label{prop:gp}
Let $H(1,x)$ be a halfspace with $\dhat{1, x} \ge 35 \delta$. 
Then there is a constant $\Kgp$, which only depends on $\delta$, such for
any pair of points $a$ and $b$ in $\overline{H(1,x)}$, the Gromov
product $\gp{a}{b} \ge \tfrac{1}{2} \dhat{1, x} - \Kgp$.
\end{proposition}

\begin{proof}
Let $a_1$ and $a_2$ be points in $H(1,x)$, and let $p_i$ be a \npp of
$a_i$ to $[1,x]$. The \npp of the path $[a_1, p_1] \cup [p_1, p_2]
\cup [p_2, a_2]$ to $[1,x]$ is contained in a $3 \delta$-neighbourhood
of $H(1, x)$. By thin triangles, a geodesic $[a_1, a_2]$ is contained
in a $2 \delta$-neighbourhood of the \npp path, and so as \npp is
coarsely distance reducing, by \pref{double}, the \npp of $[a_1, a_2]$
to $[1, x]$ is contained in a $29 \delta$-neighbourhood of $H(1, x)$.
Therefore, by \pref{half}, $[a_1, a_2]$ is contained in $H(1, y)$, for
$y$ in $[1,x]$ with $\dhat{1, y} = \dhat{1, x} - 35 \delta$, and such
a point $y$ exists as $\dhat{1,x} \ge 35 \delta$. This
implies that $\dhat{1, [a_1, a_2]} \ge \half \dhat{1, y} = \half
\dhat{1, x} - 18 \delta$, and so $\gp{a_1}{a_2} \ge \half \dhat{1, x}
- 19 \delta$.

If $b$ is a point in $\overline{H(1,x)} \cap \d X$, then there is a
sequence $b_n$ of points in $H(1,x)$ converging to $b$, and so by the
definition of Gromov product for boundary points, this implies that
$\gp{a}{b} \ge \tfrac{1}{2} \dhat{1,x} - 19 \delta$ for all points $a$
and $b$ in $\overline{H(1,x)}$. Therefore we may choose $\Kgp = 19
\delta$, which only depends on $\delta$.
\end{proof}

Let $\a$ be a \qg, and let $p$ be the nearest point on $\a$ to $1$.
We now show that for any pair of points $x_1$ and $x_2$ on $\a$ which
lie on either side of $p$, and are sufficiently far apart, then we may
choose nested halfspaces $H(1, x_1) \subset H(1, y_1)$ and $H(1, x_2)
\subset H(1, y_2)$, such that $\Hbar{1, y_1}$ and $\Hbar{1, y_2}$ are
disjoint, and any \qg $\b$ with one endpoint in each of the inner
halfspaces, $\Hbar{1, x_1}$ and $\Hbar{1, x_2}$, is contained in the
union of the outer halfspaces, and a bounded neighbourhood of $\a$.
Furthermore, the \npp of the complement of the outer two halfspaces to
$\b$ has diameter roughly $\half \dhat{x_1, x_2}$.  This is
illustrated below in Figure \ref{picture5}.

\begin{figure}[H] 
\begin{center}
\epsfig{file=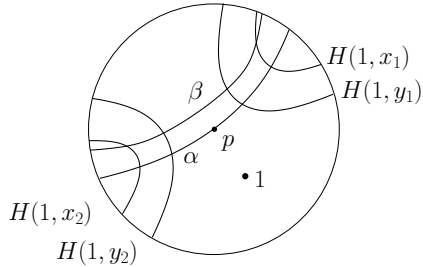, height=100pt}
\end{center}
\caption{Nested halfspaces.}\label{picture5}
\end{figure}

\begin{proposition}  \label{prop:disjoint}
Let $\a$ be a $K$-\qg, and let $p$ be a nearest point on $\a$ to $1$.
Then there are constants $A$ and $B$, which only depend on $K$ and
$\delta$, such that for any pair of points $x_1$ and $x_2$ on $\a$,
which are separated by $p$, with $\dhat{p, x_i} \ge \dhat{1,p} + A$,
then if $H(1, x_1) \subset H(1, y_1)$ and $H(1, x_2) \subset H(1,
y_2)$, are $B$-nested halfspaces, then $\Hbar{1,y_1}$ and $\Hbar{1,
  y_2}$ are disjoint, and any $K$-\qg $\b$, with one endpoint in
$\Hbar{1,x_1}$, and the other in $\Hbar{1, x_2}$, is contained in the
union of $H(1, y_1)$, $H(1, y_2)$, and a $B$-neighbourhood of $[x_1,
x_2]$.

Furthermore, there is a constant $C$, which only depends on $K$ and
$\delta$, such that the \npp of the complement of $H(1,y_1) \cup
H(1,y_2)$ to $\b$ has diameter at most $\half \dhat{x_1, x_2}
+ C$.
\end{proposition}

\begin{proof}
Let $x_1$ and $x_2$ be points on the \qg $\a$, where are separated by
$p$, which is a closest point on $\a$ to the basepoint $1$. We shall
assume that $A \ge 2 \Knpphalf + 66 \delta$, where $\Knpphalf$ is the
constant from \pref{npphalf}.
We start by finding a lower bound for the distance between the \npps
of the closures of the halfspaces $\Hbar{1, x_i}$ to a geodesic $[x_1,
x_2]$. Let $p_1$ be the \npp of a point $z_1$ in $\Hbar{1, x_1}$ to
$[1, x_1]$, let $q_1$ be the \npp of $z_1$ to the path $[1, p] \cup
[p, x_1]$. The \npp of the halfspace $H(1, x_1)$ to $[1, x_1]$ lies in
a bounded neighbourhood of $H(1, x_1)$, by \pref{npphalf}, which gives
the following estimate.
\begin{align}
\dhat{1, p_1} & \ge \half \dhat{1, x_1} - \Knpphalf \notag
\intertext{%
The \npp path $[1, p] \cup [p, x_1]$ is contained in a $6
\delta$-neighbourhood of $[1,x_1]$, by \pref{double}. As \npps to
close paths are close, $\dhat{p_1, q_1} \le 24 \delta$, by
\pref{close}.  This implies%
}
\dhat{1, q_1} & \ge \half \dhat{1, x_1} - \Knpphalf - 24 \delta. \notag
\intertext{%
Nearest point projection paths are almost geodesic, so $\dhat{1, x_1}
\ge \dhat{1, p} + \dhat{p, x_1} - 24 \delta$, by \pref{double}, and we
have assumed $\dhat{p, x_1} \ge \dhat{1,p} + A$. Therefore%
}
\dhat{1, q_1} & \ge \dhat{1, p} + \half A - \Knpphalf - 33 \delta. \notag
\intertext{%
Therefore, as we have assumed that $A \ge 2 \Knpphalf + 66 \delta$,
the point $q_1$ lies on
$[x_1, p]$, not on $[1, p]$, and $\dhat{p, q_1} \ge \half A -
\Knpphalf - 33 \delta$. A similar argument shows that as $A \ge 2
\Knpphalf + 66 \delta$, for any $z_2 \in H(1, x_2)$, the \npp
$q_2$ of $z_2$ to $[1, p] \cup [p, x_2]$ lies in $[p, x_2]$ and
$\dhat{p, q_2} \ge \half A - \Knpphalf - 33 \delta$.  As $q_1, q_2$
and $p$ lie on a common geodesic $[x_1, x_2]$, this implies that%
}
\dhat{q_1, q_2} & \ge A - 2 \Knpphalf - 66 \delta. \label{eq:qbound}
\end{align}
Therefore, if $A \ge 2 \Knpphalf + \Knested + 66 \delta$, then the
\npps of $\Hbar{1, x_1}$ and $\Hbar{1,x_2}$ to $[x_1, x_2]$ are
distance at least $\Knested$ apart, and so in particular they have
disjoint closures, by \pref{nested}.

Let $\b$ be a $K$-\qg, with endpoints $\k_i \in H(1, x_i)$, and let
$q_i$ be the \npp of $\k_i$ to $[x_1, x_2]$. Recall that by \pref{npp
  qgs}, there are constants $K'$ and $L$, which only depend on $K$ and
$\delta$. such that if $\dhat{q_1, q_2} \ge K'$, then for all $n$
sufficiently large, the subsegment of $\b$ between $\b_n$ and
$\b_{-n}$ is contained in an $L$-neighbourhood of the \npp path
$[\b_n, q_1] \cup [q_1, q_2] \cup [q_2, \b_{-n}]$. As $\k_i \in
\Hbar{1, x_i}$, the distance between $q_1$ and $q_2$ satisfies the
inequality in line \eqref{eq:qbound} above, so if $A \ge K' + 2
\Knpphalf + 66 \delta$ then $\dhat{q_1, q_2} \ge K'$. The \npp of
$[\b_n, q_i]$ to $[1, x_1]$ is contained in a $(2 \Knpphalf + 24
\delta )$-neighbourhood of $H(1, x_1)$, and is therefore contained in
$H(1, y_1)$, for $y_1$ on $[1, x_1]$ with $\dhat{1, y_1} = \dhat{1,
  x_1} - 3 \Knpphalf - 24 \delta$, and a similar argument applies to
$[q_2, \b_{-n}]$. Any part of $\b$ outside these halfspaces is
contained in an $L$-neighbourhood of $[q_1, q_2]$, which is a subset
of $[x_1, x_2]$.  Therefore we may choose $A = K' + 2 \Knpphalf +
\Knested + 66 \delta$, and $B = L + 3 \Knpphalf + 24 \delta$, which
only depend on $K$ and $\delta$, as required.

We now find an upper bound for the size of the \npp of the complement
of $H(1, x_1) \cup H(1, x_2)$ to $[x_1, x_2]$. Let $z_1 \in \Hbar{x_1,
  1}$, let $p_1$ be the \npp of $z_1$ to $[1,x_1]$, and let $q_1$ be
the \npp of $z_1$ to $[1, p] \cup [p, x_1]$. Then, by \pref{npphalf},
as the \npp of a halfspace $\Hbar{x_1, 1}$ lies in a bounded
neighbourhood of the halfspace,
\begin{align*} 
\dhat{1, p_1} & \le \half \dhat{1, x_1} + \Knpphalf.
\intertext{%
A \npp path is almost a geodesic, by \pref{double}, so $[1, x]$ lies
in a $6 \delta$-neighbourhood of $[1, p] \cup [p,x_1]$. Therefore by
\pref{close}, $\dhat{p_1, q_1} \le 24 \delta$. This implies%
}
\dhat{1, q_1} & \le \half \dhat{1, x_1} + \Knpphalf + 24 \delta. 
\intertext{%
By \pref{double}, $\dhat{1, p} + \dhat{p, q_1} - 24 \delta \le \dhat{1,
  q_1}$, which implies,%
}
\dhat{1, p} + \dhat{p, q_1} & \le \half \dhat{1, x_1} + \Knpphalf + 48 \delta.
\intertext{%
Combining this with a similar for formula for $q_2$ and the triangle
inequality, we obtain,%
}
\dhat{q_1, q_2} & \le \half \dhat{x_1, x_2} + \Knpphalf + 96 \delta.
\end{align*}
Therefore we may choose $C = \Knpphalf + 96 \delta$, which only
depends on $\delta$, as required.
\end{proof}

%%%%%%%%%%%%%%%%%%%%%%%%%%%%%%%%%%%%%%%%%%%%%%%%%%%%%%%%%%%%%%%%%%%%%%%%%%%%%%%
\subsection{Halfspace approximations} \label{section:halfspace}
%%%%%%%%%%%%%%%%%%%%%%%%%%%%%%%%%%%%%%%%%%%%%%%%%%%%%%%%%%%%%%%%%%%%%%%%%%%%%%%

Given a subset $T$ of $X \cup \d X$, we may consider all the
halfspaces which intersect the limit set of $T$ in the Gromov
boundary, and which are at least a certain distance from the
origin. To be precise, given a subset $T$ of $X \cup \d X$, define
$N_r(T)$ to be the closure of $\bigcup H(1,x)$, where the union runs
over all $x$ with $\nhat{x} \geqslant r$ and $\overline{H(1,x)} \cap
\overline{T} \cap \d X \not = \varnothing$. Roughly speaking we may
think of $N_r(T)$ as a regular neighbourhood of $\overline{T} \cap \d
X$ in $X \cup \d X$. The sets $N_r(T)$ are a decreasing sequence, as
$N_r(T) \subseteq N_s(T)$ for all $r > s$.  We now show that the
intersection of the sets $N_r(T)$ is actually equal to $T$.

\begin{lemma} \label{lemma:intersections}
Let $T$ be a closed subset of the Gromov boundary $\d X$. Then
$\bigcap N_r(T) = T$.
\end{lemma}

\begin{proof}
The set $T$ is contained in $N_r(T)$ for all $r$, so $T$ is a subset
of $\bigcap N_r(T)$. We now show the reverse inclusion.  Let $\l$ be a
point in $\bigcap N_r(T)$. Then there is a sequence $x_n$ such that
$\l \in \overline{H(1,x_n)}$, with $\nhat{x_n} \to \infty$, and
$\overline{H(1,x_n)} \cap T \not = \varnothing$. Choose a sequence
$\l_n \in \overline{H(1,x_n)} \cap T$, then as the Gromov product of
two points in $\overline{H(1,x)}$ is bounded below by roughly $\half
\dhat{1,x}$, by \pref{gp}, as we may assume that $\dhat{1, x_n} \ge 35 \delta$, as
$\dhat{1, x_n} \to \infty$. This implies that $\gp{\l}{\l_n} \ge
\tfrac{1}{2} \dhat{1,x_n} - \Kgp$, where $\Kgp$ only depends on
$\delta$, and so $\gp{\l}{\l_n} \to \infty$ and the sequence $\l_n$
converges to $\l$. As $T$ is closed, this implies that $\l \in T$, and
so $\bigcap N_r(T)$ is a subset of $T$, and hence equal to $T$, as
required.
\end{proof}

We now show that an $N_r$-neighbourhood of $N_r(T)$ is contained
in $N_{r-\Knbd}(T)$, for some constant $\Knbd$ which only depends on
$\delta$.

\begin{lemma} \label{lemma:neighbourhood}
There is a constant $\Knbd$, which only depends on $\delta$, such that
$N_r(N_r((T))) \subset N_{r-\Knbd}(T)$, as long as $r \ge \Knbd$.
\end{lemma}

\begin{proof}
Let $\overline{H(1,x)}$ and $\overline{H(1,y)}$ be a pair of
intersecting halfspaces, with $\dhat{1,x} \ge r$ and $\dhat{1,y} \ge
r$. It suffices to show that there is a $\Knbd$ such that
$\overline{H(1,y)} \subset \overline{H(1,x')}$, for any
$\overline{H(1,y)}$ which intersects $\overline{H(1,x)}$, where $x'$
lies on $[1,x]$ with $\nhat{x'} = \nhat{x} - \Knbd$.

We first show that there is a point $z \in X$ with bounded \npp to
both $[1,x]$ and $[1,y]$. Let $p$ be the \npp of $z$ to $[1,x]$, and
let $q$ be the \npp of $z$ to $[1,y]$.

\begin{claim}
There is a point $z$, and a constant $K$, which only depends on
$\delta$, such that $\dhat{1, p} \ge \half r - K$ and $\dhat{1, q} \ge
\half r - K$.
\end{claim}

\begin{proof}
Choose $t \in \overline{H(1,x)} \cap \overline{H(1,y)}$. If $t$ lies
in $X$, then choose $z$ to be $t$, and $\dhat{1, p} \ge \half
\dhat{1,x} - 3 \delta$ by \pref{half}, and similarly $\dhat{1,
  q} \ge \half \dhat{1,y} - 3 \delta$. If $t$ lies in $\d X$,
then let $t_n$ be a sequence which converges to $z$. Then by
\pref{npphalf}, $\dhat{1, p} \ge \half \dhat{1,x} - \Knpphalf$,
and similarly $\dhat{1, q } \ge \half \dhat{1, x} - \Knpphalf$,
and by \pref{npp boundary}, we may pass to a subsequence such that
$\dhat{1, p} \ge \half \dhat{1,x} - \Knpphalf - 15 \delta$,
and similarly $\dhat{1, q} \ge \half \dhat{1, x} - \Knpphalf
- 15 \delta$. Choose $z$ to be $t_n$ for some $n$. We have assumed that
$\dhat{1,x} \ge r$ and $\dhat{1,y} \ge r$, so we may choose $K$ to be
$18 \delta + \Knpphalf$, which only depends on $\delta$, as
required.
\end{proof}

By \pref{double}, the nearest point projection paths
$[1,p] \cup [p,z]$ and $[1,q] \cup [q,z]$ are both $\Knpp$-\qgs, for
some $\Knpp$ which only depends on $\delta$. Therefore, there is an
$L$, which only depends on $\delta$, such that the two \qgs are
contained in a $L$-neighbourhood of the geodesic $[1,z]$. Let $a$ be a
point in $\overline{H(1,y)}$, and let $a_q$ be the nearest point
projection of $a$ onto $[1,q] \cup [q,z]$. If $a_q$ lies in $[1,q]$,
then $\dhat{1,a_q} \ge \half \dhat{1,y} - \Knpphalf$, by
\pref{npphalf}.  If $a_q$ lies in $[q,z]$, then as $[1,q] \cup [q,z]$
is a closest point projection path, then $\dhat{1,a_q} \ge \dhat{1,q}
- 24 \delta$, by \pref{double}.  Therefore, using the claim above,
\begin{align}
\dhat{1,a_q} & \ge \half r - K - \Knpphalf - 24 \delta. \notag  
\intertext{%
As closest point projection onto close paths is close, \pref{close},
if $a_z$ is the \npp of $a$ onto $[1,z]$, then $\dhat{a_q,a_z} \le 3L + 6
\delta$. Similarly, if $a_p$ is the \npp of $a$ onto $[1,p] \cup [p,z]$,
then $\dhat{a_p,a_z} \le 3L + 6 \delta$. This implies that%
}
\dhat{1,a_p} & \ge \half r - K - \Knpphalf - 6L - 36 \delta. \label{eq:ap bound}
\intertext{%
Let $a_x$ be the \npp of $a$ to $[1,x]$. If $a_x$ lies in $[p, x]$,
then%
}
\dhat{1,a_x} & \ge \half r - \Knpphalf,  \label{eq:ax bound 1}
\intertext{%
by \pref{npphalf}. Otherwise, $a_x$ lies in $[1,p]$. If $a_p$ lies on
$[1,p]$, then we may choose $a_x = a_p$, so we obtain the same bound
for $\dhat{1, a_x}$ as the bound for $\dhat{1, a_p}$ in line
\eqref{eq:ap bound}. Finally, suppose $a_p$ lies on $[p,z]$. Consider
the \npp path $[a, a_x] \cup [a_x, p] \cup [p, a_p]$. If $\dhat{p,
  a_x} \ge 15 \delta$, then the geodesic $[a, a_p]$ passes within $6
\delta$ of $a_x$, which by \pref{double}, implies that $a$ is closer
to $a_x$ than $a_p$, which contradicts the fact that $a_p$ is a \npp
of $a$ to $[1, p] \cup [p,z]$. If $\dhat{p, a_x} \le 15 \delta$, then%
}
\dhat{1,a_x} & \ge \half r - \Knpphalf - 15 \delta. \label{eq:ax bound 2}
\intertext{%
Therefore, comparing the bounds in lines \eqref{eq:ap bound},
\eqref{eq:ax bound 1} and \eqref{eq:ax bound 2}, we have shown that%
}
\dhat{1, a_x} & \ge \half r - K - \Knpphalf - 6L - 36 \delta. \notag
\end{align}
Now using \pref{nested}, there is a halfspace $H(1,x')$ such that
$\overline{H(1,y)} \subset \overline{H(1,x')}$, where $x'$ lies on
$[1,x]$ with $\dhat{1,x'} = \dhat{1,x} - \Knbd$, where $\Knbd =
\Knested + 2(K + \Knpphalf + 6L + 36 \delta)$, which only depends on
$\delta$, as required.
\end{proof}

%%%%%%%%%%%%%%%%%%%%%%%%%%%%%%%%%%%%%%%%%%%%%%%%%%%%%%%%%%%%%%%%%%%%%%%%%%%%%%%
\subsection{Hyperbolic isometries}
%%%%%%%%%%%%%%%%%%%%%%%%%%%%%%%%%%%%%%%%%%%%%%%%%%%%%%%%%%%%%%%%%%%%%%%%%%%%%%%

Let $g$ be an isometry of a $\delta$-hyperbolic space $X$. We shall
write $\t_g$ for the \emph{translation length} of $g$, which is
\[ \t_g = \lim_{n \to  \infty}\tfrac{1}{n}\dhat{x, g^n x}. \]
This definition is independent of the point $x$, and by the triangle
inequality, $\t_g \le \dhat{x, gx}$, for any point $x \in X$. We say
$g$ is a \emph{hyperbolic isometry} if $\t_g > 0$.  A
\emph{quasi-axis} for a hyperbolic element is a bi-infinite $K$-\qg
$\a$ which is coarsely invariant under $g$, i.e.  $\a$ and $g^k \a$
are Hausdorff distance $2L$ apart, for any $k$, where $L$ is a
geodesic neighbourhood constant for $\a$.  For example, the images of
any point $x \in X$ under powers of $g$, $\{ g^n x \}$, forms a
$K$-quasi-axis for $g$, where $K$ depends on $g$ and $x$, and the ends
of the quasi-axis are the unique pair of fixed points for $g$ in the
Gromov boundary. We will refer to $\l^+(g) = \lim_{n \to \infty} g^n
x$ as the \emph{stable fixed point} for $g$, and we will refer to
$\l^-(g) = \lim_{n \to \infty} g^{-n} x$ as the \emph{unstable fixed
  point} for $g$. We now show that if the translation length of the
isometry $g$ is sufficiently large, then the stable fixed point
$\l^+(g)$ lies in the halfspace $\Hbar{x, g x}$.

\begin{proposition} \label{prop:fixed point}
There is a constant $\Kpa$, which only depends on $\delta$, such that
if $g$ is a hyperbolic isometry with translation length at least
$\Kpa$, then the stable fixed point of $g$ lies in
$\overline{H(x,g x)}$, and the unstable fixed point of $g$ lies
$\overline{H(x,g^{-1} x)}$.
\end{proposition}

\begin{proof}
It suffices to show that if $\t_g$ is sufficiently large then $g^n x
\in H(x, g x)$, for all $n$ sufficiently large.  Consider the geodesic
$\g_n = [g^{-n} x, g^n x]$ and its image under $g$, the geodesic $g
\g_n = [g^{-n+1} x, g^{n+1} x]$. This pair of geodesics has endpoints
distance $\dhat{x, g x}$ apart, so by thin triangles, the subsegments
of the geodesics distance greater than $2 \dhat{x, g x} + \delta$ from
their endpoints are contained within $2 \delta$-neighbourhoods of each
other.

Let $p$ be a closest point to $x$ on $\g_n$, and therefore $gp$ is a
closest point to $g x$ on $g \g_n$. Let $q$ be a closest point to $g
x$ on $\g_n$. As subsegments of $\g_n$ and $g \g_n$ close to $g x$ are
contained within $2 \delta$-neighbourhoods of each other, the nearest
point projections of $g x$ to these two geodesics are also close
together, by \pref{close}, which implies $\dhat{g p, q} \le 18
\delta$. The translation distance $\t_g$ is a lower bound for the
distance $g$ moves any point, so this implies that
\begin{equation} \label{eq:tg} 
\dhat{p, q} \ge \t_g - 18 \delta. 
\end{equation}
The triangle inequality implies
\begin{align}
\dhat{g x, g^n x} & \le \dhat{g x, q} + \dhat{q, g^n x}, \label{eq:fixed point
  a}
\intertext{%
and \pref{double} implies that the \npp path $[x, p] \cup [p, g^n x]$ is
almost a geodesic, i.e.%
}
& \dhat{x, p} + \dhat{p, g^n x} -24 \delta \le \dhat{x, g^n x}. \notag
\end{align}
We have shown that $\dhat{gp, q} \le 18 \delta$, and as $g$ is an
isometry, $\dhat{x, p} = \dhat{g x, gp}$. Furthermore, as $p, q$ and
$g^n x$ lie on a common geodesic, $\dhat{p, g^n x} = \dhat{p, q} +
\dhat{q, g^n x}$. This implies
\begin{equation} \label{eq:fixed point b}
\dhat{g x, q} + \dhat{q, g^n x} + \dhat{p,q}  - 42 \delta \le \dhat{x,
  g^n x}. 
\end{equation}
Comparing lines \eqref{eq:fixed point a} and \eqref{eq:fixed point b}
above shows that if $\dhat{p, q} \ge 42 \delta$ then $g^n x \in H(x, g
x)$, for all $n$ sufficiently large, and so $\l^+(g) \in \Hbar{x, g
  x}$. Using line \eqref{eq:tg} above shows that if $\t_g \ge 60
\delta$, then $\dhat{p, q} \ge 42 \delta$, so we may choose $\Kpa =
60\delta$.  Replacing $g$ by $g^{-1}$ in the argument above shows that
if the relative translation length of $g$ is at least $\Kpa$, then
$\l^-(g) \in \overline{H(x,g^{-1} x)}$.
\end{proof}

We now show that we may estimate the translation length of of a
hyperbolic isometry $g$ in terms of the distance it moves a point $x
\in X$, and the distance from $x$ to a quasi-axis for $g$.

\begin{proposition} \label{prop:translation}
Let $g$ be a hyperbolic isometry with a $K$-quasi-axis $\a$.  Then
there is a constant $\Ktrans$, which only depends on $K$ and $\delta$,
such that if $\t_g \ge \Ktrans$, then 
\[\norm{ \t_g -  \dhat{x, g x} + 2\dhat{x,\a} } \le \Ktrans. \]
\end{proposition}

\begin{proof}
We start by estimating the translation length $\t_g$ in terms of the
distance $g$ moves a point $p$ on a $K$-quasi-axis $\a$ for $g$. Let $L$ be
a geodesic neighbourhood constant for the \qg $\a$, and and recall
that $L$ depends only on $K$ and $\delta$. We are interested in
isometries with sufficiently large translation lengths, and we will
assume $\t_g \ge 10 L$. 

We now find a lower bound for $\t_g$ in terms of the distance $g$
moves a point on the quasi-axis $\a$. Let $p$ be a point on $\a$, and
consider the geodesic $[p,g^n p]$. Let $q_k$ be the nearest point
projection of $g^k p$ to $[p, g^n p]$, for $0 \le k \le n$. The point
$g^k p$ lies on $g^k \a$, which is contained in a $2 L$-neighbourhood
of $\a$, so the distance from $g^k p$ to $q_k$ is at most $3L$. The
distance between $g^k p$ and $g^{k+m} p$ is equal to $\dhat{p,g^m p}$,
which implies that
\begin{equation} \label{eq:bounds} 
\dhat{p, g^m p} - 6 L \le \dhat{ q_k, q_{k+m} } \le \dhat{p,g^m p}
+ 6 L, 
\end{equation}
for all $ 0 \le k \le k + m \le n$. The distance $g$ moves any point
is an upper bound for $\t_g$, and $\t_{g^{m}} = m \t_g$, which implies
\begin{equation} \label{eq:intervals}
m \t_g- 6 L \le \dhat{ q_k, q_{k+m} },  
\end{equation}
again for all $0 \le k \le k + m \le n$. The nearest \npps $q_k$ lie
on a common geodesic, and we now show that if $\t_g$ is sufficiently
large, then they appear on the geodesic in the order $q_0, q_1, \ldots
q_{n}$. If they appear on the geodesic in a different order, then
there must be a triple $q_k, q_{k+1}, q_{k+2}$ such that $q_{k+1}$
does not separate $q_k$ and $q_{k+2}$. Line \eqref{eq:bounds} then
implies that $\dhat{q_k, q_{k+2}} \le 12 L$, and , and then line
\eqref{eq:intervals} with $m=2$ implies that $\t_g \le 9L$, which
contradicts the fact that we assumed $\t_g \ge 10 L$. As the $q_i$ lie
in consecutive order on a geodesic, this implies that
\[ \dhat{p, g^n p} = \dhat{q_0, q_1} + \dhat{q_1, q_2} + \cdots +
\dhat{q_{n-1}, q_n}. \]
Using line \eqref{eq:bounds} with $m=1$, this implies $\dhat{p, g^n p}
\ge n(\dhat{p,gp} - 6 L)$, and this holds for all $n$, so the
translation length $\t_g$ is at least $\dhat{p,gp} - 6 L$.  Therefore
we have shown
\begin{equation} \label{eq:trans bound}
\dhat{p, gp} - 6 L \le \t_g \le \dhat{p, gp}, 
\end{equation}
for any point $p$ which lies on a $K$-quasi-axis for $g$.

\begin{figure}[H] 
\begin{center}
\epsfig{file=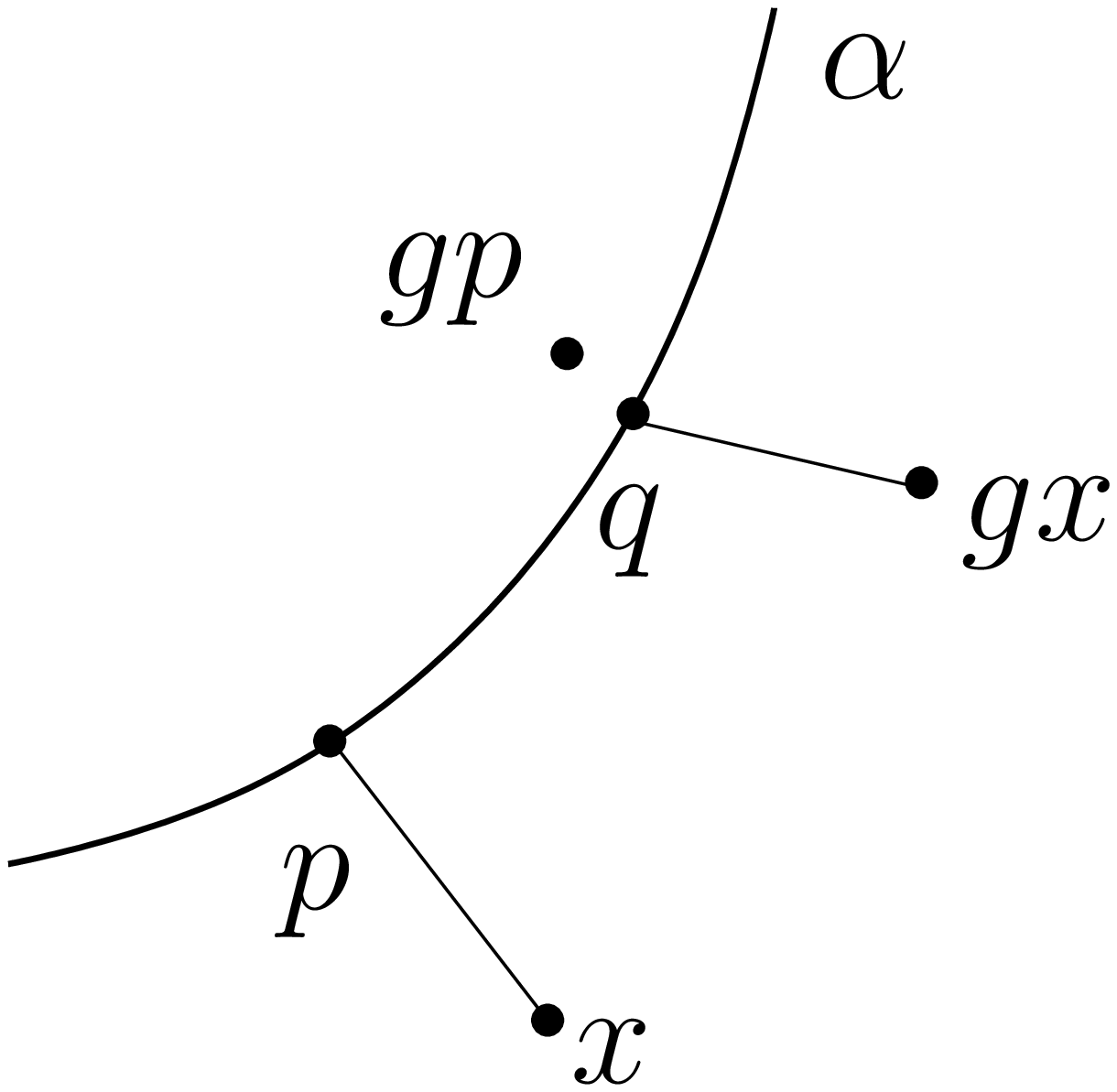, height=100pt}
\end{center}
\caption{The translation length $\t_g$ is roughly $\dhat{x, gx} -
  2\dhat{x, \a}$.} \label{picture29}
\end{figure}

We now show that if $p$ is a \npp of $x$ to $\a$, and $q$ is a \npp of
$g x $ to $\a$, then $q$ is a bounded distance from $gp$.  The
distance from $x$ to $p$ is the same as the distance from $g x$ to
$gp$, and the point $gp$ lies on $g \a$, which lies in a $2L$
neighbourhood of $\a$. The point $q$ is a nearest point on $\a$ to $g
x$, and the point $gp$ is a nearest point on $g \a$ to $g x$, so the
distance from $g p$ to $q$ is at most $6 L + 6\delta$, as nearest
point projections onto close paths are close, by \pref{close}. This is
illustrated in Figure \ref{picture29} above. Using line \eqref{eq:trans
  bound} above, this implies
\begin{equation} \label{eq:trans bounda}
\dhat{p, q} - 12 L - 6 \delta \le \t_g \le \dhat{p, q} + 6 L + 6 \delta. 
\end{equation}
Let $p'$ be the \npp of $x$ to $[p,q]$, and let $q'$ be the \npp
projection of $g x$ to $[p,q]$.  The segment of $\a$ between $p$ and
$q$ lies in an $L$-neighbourhood of $[p, q]$, so as \npps to close
paths are close, \pref{close}, $\dhat{p, p'} \le 6 L + 6 \delta$, and
$\dhat{q, q'} \le 6 L + 6 \delta$. Using line \eqref{eq:trans bounda}
above, this implies
\begin{equation} \label{eq:trans bound b}
\dhat{p', q'} - 18 L - 12 \delta \le \t_g \le \dhat{p', q'} + 12 L +
12 \delta. 
\end{equation}
If $\t_g \ge 18 L + 27 \delta$, then $\dhat{p', q'} \ge 15 \delta$, so
the \npp path $[x, p'] \cup [p', q'] \cup [q', g x ]$ is almost a
geodesic, by \pref{double}, i.e.
\[ \dhat{x, p'} + \dhat{p', q'} + \dhat{q', g x} - 24 \delta \le
\dhat{x, g x} \le \dhat{x, p'} + \dhat{p', q'} + \dhat{q', g x}. \]
Together with line \eqref{eq:trans bound b}, and the bounds on
$\dhat{p, p'}$ and $\dhat{q, q'}$, this implies that
\[ \dhat{x, g x} - 2 \dhat{x, \a} - 30 L - 24 \delta \le \t_g \le
\dhat{x, g x} - 2 \dhat{x, \a} + 24 L + 48 \delta. \]
Therefore we may choose $\Ktrans = 30 L + 48 \delta$, which only
depends on $K$ and $\delta$, as required.
\end{proof}

%%%%%%%%%%%%%%%%%%%%%%%%%%%%%%%%%%%%%%%%%%%%%%%%%%%%%%%%%%%%%%%%%%%%%%%%%%%%%%%
\section{The disc set has harmonic measure zero}
\label{section:disc set}
%%%%%%%%%%%%%%%%%%%%%%%%%%%%%%%%%%%%%%%%%%%%%%%%%%%%%%%%%%%%%%%%%%%%%%%%%%%%%%%

Fix an identification of the closed orientable surface $\S$ with the
boundary of a handlebody. The \emph{disc set} $\ds$ is the subset of the
complex of curves consisting of all simple closed curves which bound
discs in the handlebody. In this section we show that the disc set has
harmonic measure zero

\begin{theorem} \label{theorem:disc set}
Let $(G,\mu)$ be a random walk on the mapping class group of an
orientable surface of finite type, which is not a sphere with three or
fewer punctures, such that the semi-group generated by $\mu$ contains
a complete subgroup. Then the harmonic measure of the disc set is zero. 
\end{theorem}

In order to show this, we will use the following theorem of Kerckhoff
\cite{kerckhoff}.

\begin{theorem} \cite{kerckhoff}*{Proposition on p36} 
\label{theorem:kerckhoff}
There is a recurrent maximal train track $\s$, such that for any
handlebody, $\s$ can be split at most $-9\chi$ times to a recurrent
maximal train track
$\s'$ such that $N(\s')$ is disjoint from the disk set of the handlebody. Here $\chi$
is the Euler characteristic of the boundary of the handlebody.
\end{theorem}

In fact we shall show that any set with the splitting property
described in the theorem above has harmonic measure zero.

\begin{theorem}\label{theorem:split zero}
Let $(G,\mu)$ be a random walk on the mapping class group of an
orientable surface of finite type, which is not a sphere with three or
fewer punctures, such that the semi-group generated by $\mu$ contains
a complete subgroup.

Let $X$ be a set such that there is a train track $\t$ and a constant
$N$, such that for any image $gX$ of $X$, $\t$ may be split at most
$N$ times to produce a train track $\t'$ such that $N(\t')$ is
disjoint from $gX$. Then the harmonic measure of $X$ is zero.
\end{theorem}

Recall that a subgroup of the mapping class group is \emph{complete}
if the endpoints of its \pA elements are dense in $\PMF$. We now show
that for a complete subgroup, pairs of endpoints $(\l^+(g), \l^-(g))$
are dense in $\PMF \cross \PMF$.

\begin{lemma}
Let $H$ be a complete subgroup of the mapping class group. Then pairs
of endpoints $(\l^+(g),\l^-(g))$ of \pA elements $g$ in $H$ are dense
in $\PMF \cross \PMF$.
\end{lemma}

\begin{proof}
Let $f$ and $g$ be a pair of \pA elements which generate a Schottky
group $F$ with limit set $\Lambda(F)$. Recall that a Schottky group is
a generated by a pair of \pA elements $f$ and $g$, which have
associated disjoint halfspaces $F_+, F_-, G_+$ and $G_-$, such that
$f(F_-^c) \subset F_+$, $f^{-1}(F^c_+) \subset F_-$, $g(G_-^c) \subset G_+$
and $g^{-1}(G^c_+) \subset G_-$, where $X^c$ denotes the complement of $X$. A
standard ping-pong argument shows that pairs of endpoints $(\l^+(h),
\l^-(h))$ are dense in $\Lambda(F) \cross \Lambda(F)$, as $h$ runs
over elements of $F$. Given $(\a,\b) \in \PMF \cross \PMF$, let $a_i$
be a sequence of \pA elements such that $\l^+(a_i)$ tends to $\a$, and
let $b_i$ be a sequence of \pA elements such that $\l^-(b_i)$ tends to
$\b$. Possibly after passing to a subsequence and taking powers, we
may assume that the pairs $\{ a_i, b_i \}$ generate Schottky groups
$F_i$, with limit sets $\Lambda(F_i)$. Furthermore, $\a$ and $\b$ lie
in the limit set of the limit sets $\Lambda(F_i)$, so we may choose a
sequence of \pA elements $h_i$ such that $\l^+(h_i)$ tends to $\a$ and
$\l^-(h_i)$ tends to $\b$.
\end{proof}

In particular, this implies that give any pair of train tracks $\s$
and $\t$, we can find an element $f$ in a complete subgroup such that
$f\s \prec \t$. The element $f$ may be chosen to be a suitably large
power of a \pA element with stable fixed point in $\t$, and unstable
fixed point in the complement of $\s$.

We first show that if any translate $gX$ has arbitrarily many disjoint
images $w_i gX$, where the word length $\norm{w_i}$ of the disjoint
images is bounded independent of $g$, then $X$ has harmonic measure
zero. We then use the train track condition described above to
construct such a sequence of disjoint translates.

\begin{lemma} \label{lemma:measure zero}
Let $(G,\mu)$ be a random walk on the mapping class group of an
orientable surface of finite type, which is not a sphere with three or
fewer punctures, such that the semi-group $H^-$ generated by $\rmu$ is
non-elementary. Let $\nu$ be the corresponding harmonic measure
on the boundary $\fmin$.

Let $X$ be a set with the property that there is a sequence $\{ k_i
\}_{i \in \N}$ such that for any translate $gX$ of $X$ there is a
sequence $\{ w_i\}_{i \in \N}$ of elements of $H^-$, such that the
translates $gX, w_1gX, w_2gX, \ldots$ are all disjoint, and
$\norm{w_i} \le k_i$ for all $i$. Then $\nu(X) = 0$.
\end{lemma}

The sequence of group elements $\{ w_i \}$ may depend on $g$, but the
sequence of numbers $\{ k_i \}$ does not.  We emphasize that the
lengths of the $w_i$ are bounded in the word metric on $G$, not just
relatively bounded, as we make essential use of this in the proof. The
idea of the proof is to consider a translate $gX$ which is very close
to the supremum. As the measure $\nu$ is $\mu$-stationary, if
$\mu(h^{-1}) = \rmu(h) \not = 0$, then $\nu(hgX)$ must also be close
to the supremum. So if for any $gX$ we can produce a collection of
disjoint translates $w_i gX$, where the lengths of the $w_i$ are
bounded independently of $g$, then by choosing $gX$ arbitrarily close
to the supremum, we can force the $w_i gX$ to also be close to the
supremum.  However, this implies that the total measure of the
disjoint translates $w_i gX$ is bigger than one, a contradiction.

\begin{proof}
Suppose that $\sup \{ \nu(gX) \mid g \in G\} = s > 0$, and choose $N >
2/s$. For each $h \in H^-$, let $n_h$ be the smallest $n$ such that
$\rmun{n}(h) > 0$. Let $M = \max \{ 1/\rmun{n_h}(h) \mid h \in H^-, \norm{h}
\leqslant k_N\}$, and set $\e = 1/(MN)$. Finally, choose $g$ such that
the harmonic measure of $gX$ is within $\e$ of the supremum,
i.e. $\nu(gX) \ge s - \e$, and let $w_igX$ be a sequence of $N$ disjoint
translates with $\norm{w_i} \le k_i$ for each $i$ between $1$ and $N$.

The harmonic measure $\nu$ is $\mu$-stationary, and hence
$\mun{n}$-stationary for any $n$, which implies
\begin{align*} 
\nu(gX) &= \sum_{h \in G} \mun{n}(h) \nu(h^{-1}gX).
\intertext{
For notational convenience we shall write $h^{-1}$ instead
  of $h$, and we may then rewrite the formula above as
}
\nu(gX) &= \sum_{h \in G} \rmun{n}(h) \nu(hgX).
\intertext{
Let $f$ be an element of $H^-$, and choose $n$ such that
$\rmun{n}(f) > 0$. Then }
\rmun{n}(f) \nu(fgX) &= \nu(gX) \ -
\sum_{h \in G \setminus f} \rmun{n}(h) \nu(hgX).  
\intertext{%
As we have chosen $gX$ to have measure within $\e$ of the supremum,
this implies
}
\rmun{n}(f) \nu(fgX) & \ge s-\e -\sum_{h \in G \setminus f} \rmun{n}(h) \nu(hgX).
\intertext{%
The harmonic measure of each translate of $X$ is at most the supremum
$s$, so }
\rmun{n}(f) \nu(fgX) & \geqslant s - \e - s \sum_{h \in G \setminus f} \rmun{n}(h), 
\intertext{%
and the sum of $\rmun{n}(h)$ over all $h \in G \setminus f$
  is just $1-\rmun{n}(f)$, which implies that
}
\rmun{n}(f) \nu(fgX) & \geqslant s - \e - s(1-\rmun{n}(f)).
\intertext{%
Dividing by $\rmun{n}(f)$ gives the following estimate for the
harmonic measure of the translate,
}
\nu(fgX) & \geqslant s - \e /\rmun{n}(f),
\intertext{%
for any $n$ such that $\rmun{n}(f) > 0$.
In particular, this estimate holds for each of the $N$ disjoint
translates $w_i gX$, with $n = n_{w_i}$, the
smallest value of $n$ such that $\rmun{n}(w_i) > 0$. Furthermore, as we
chose $M$ such that $M \ge 1/\rmun{n_h}(h)$ for all  $h \in H^-$ with
$\norm{h} \le k_N$, this implies that
}
\nu(w_i gX) & \geqslant s - \e M.
\end{align*}
As all the translates $w_i g X$ are disjoint,
\[ \nu(w_1 gX) + \dots + \nu(w_N gX)  \geqslant N(s-\e M). \]
As we chose $N > 2/s$, and  $\e = 1/(MN)$, this implies that the
total measure $\nu(\bigcup w_igX)$ is greater than one, a contradiction.
\end{proof}

We now complete the proof of Theorem \ref{theorem:split zero}, by
using the train track property to construct disjoint translates
$w_igX$ of any translate $gX$, with a bound on the word length of the
$w_i$.

\begin{proof} (of Theorem \ref{theorem:split zero})
Let $\t$ be a train track such that for any image $gX$ of $X$, $\t$
may be split at most $N$ times to produce a train track $\t'$ which is
disjoint from $gX$.  Let $\{ \t_1, \ldots , \t_b \}$ be all possible
train tracks obtained by splitting $\t$ at most $N$ times. For each
$\t_i$ chose an element $g_i \in H^-$ such that $g_i \t$ is carried by
$\t_i$. Such an element $g_i$ exists in $H^-$ as we have assumed that
$H^-$ contains a complete subgroup of the mapping class group, and
furthermore, we may choose $g_i$ such that $g_i^{-1}$ is also in $H^-$. There
is an ``inversion'' map $h$ for $\t$ that takes the
exterior of the train track polyhedron $N(\t)$ into the interior of
$N(\t)$. We can construct such a map by first choosing an element $h'$
such that $h'\t$ is disjoint from $\t$, and then composing with a
large power of a \pA with stable endpoint in the interior of $N(\t)$,
and unstable endpoint in $N(h'\t)$. Again such an element $h$ exists in
$H^-$, as we have assumed that $H^-$ contains a complete subgroup of
the mapping class group.  Let $M$ be the maximum length in
the word metric of any of the words $g_i$ or $h$ we have chosen above.

Start with the train track $\t$, and let $gX$ be some translate of
$X$. The train track $\t$ may be split at most $N$ times to produce a
train track $\t_{i_i}$ disjoint from $gX$. There is a translate
$g_{i_1}\t$ of $\t$ carried by $\t_{i_i}$. The map $h$ takes the
exterior of $N(\t)$ to the interior of $N(\t)$, so the conjugate of
$h$ by $g_{i_1}$, namely $g_{i_1}hg_{i_1}^{-1}$ takes the exterior of
$N(g_{i_1}\t)$ to the interior of $N(g_{i_1}\t)$. This implies that
$g_{i_1}hg_{i_1}^{-1}gX$ lies in the interior of the train track
$g_{i_1}\t$, and so in particular is disjoint from $gX$.

We may continue this process, starting with the train track
$g_{i_1}\t$, and the image $g_{i_1}hg_{i_1}^{-1}gX$ of $gX$. The train
track $\t$ may be split at most $N$ times to form a train track $\t_j$
disjoint from any image of $gX$, e.g. $hg_{i_1}^{-1}gX$. This implies
that $g_{i_1}\t$ may be split at most $N$ times to form a train track
$g_{i_1}\t_{i_2}$ disjoint from $g_{i_1}hg_{i_1}^{-1}gX$. The
corresponding image of $\t$ carried by $g_{i_1}$ will be
$g_{i_1}g_{i_2}\t$. Using the corresponding conjugate of the inversion
map, $g_{i_1}g_{i_2}h(g_{i_1}g_{i_2})^{-1}$ we can construct a new
image of $gX$, namely $g_{i_1}g_{i_2}h(g_{i_1}g_{i_2})^{-1}gX$,
disjoint from the previous two.  Therefore we can inductively
construct a sequence of disjoint images of $gX$.  The $k$-th disjoint
translate is $w_k gX$ where $w_k = (g_{i_1}\ldots
g_{i_k})h(g_{i_1}\ldots g_{i_k})^{-1}$, so each $w_k$ lies in $H^-$,
and the word length of $w_k$ is at most $(2k+1)M$.
\end{proof}

%%%%%%%%%%%%%%%%%%%%%%%%%%%%%%%%%%%%%%%%%%%%%%%%%%%%%%%%%%%%%%%%%%%%%%%%%%%%%%%
\section{Independence} \label{section:independence}
%%%%%%%%%%%%%%%%%%%%%%%%%%%%%%%%%%%%%%%%%%%%%%%%%%%%%%%%%%%%%%%%%%%%%%%%%%%%%%%

In this section we show that the joint distribution in $\fmin \cross
\fmin$ of the pairs of stable and unstable foliations of sample path
locations $w_n$ which are \pA, converges to the distribution $\nu
\cross \rnu$ as $n$ tends to infinity, where $\nu$ is the harmonic
measure determined by $\mu$, and $\rnu$ is the reflected harmonic
measure, i.e. the harmonic measure determined by $\rmu$.

The argument has two main steps. First we show that the joint
distribution of sample path locations and their inverses $(w_n,
w_n^{-1})$ is asymptotically independently distributed as $\nu \cross
\rnu$. Secondly, we show that the joint distribution of stable and
unstable foliations of \pA elements $(\l^+(w_n),\l^-(w_n))$ converges
to the same limit as the limit of the distributions of
$(w_n,w_n^{-1})$.

Recall that the probability space $(\gz, \P)$ is the space of all
sample paths of the random walk, and $w_n$ is the image of $w$ under
projection onto the $n$-th factor, which gives the location of the
sample path at time $n$.  Consider the sequence of random variables $w
\mapsto (w_n, w_n^{-1})$. We shall write $X_n$ for the corresponding
measures on $\Gbar \cross \Gbar$, defined by $X_n(U) = \P( (w_n,
w_n^{-1}) \in U)$, where $\Gbar$ is the union of the relative space
with its Gromov boundary.  Note that the sequence $w_n$ converges in
$\Gbar$ almost surely, but the sequence $w_n^{-1}$ need not converge
in $\Gbar$, so convergence in distribution, or weak-$*$ convergence,
is the most that can be expected.  Furthermore, $X_n$ is not the
product $\mun{n} \cross \rmun{n}$, although the induced measures on
each factor are $\mun{n}$ and $\rmun{n}$, respectively.  However, we
will now show that the limit $\lim_{n \to \infty}X_n$ is in fact equal
to the product $\nu \cross \widetilde \nu$.

\begin{theorem} 
\label{theorem:independence}
Let $G$ be the mapping class group of an orientable surface of finite
type, which is not a sphere with three or fewer punctures. Let $\mu$
be a finitely supported probability distribution on $G$, whose support
generates a non-elementary subgroup of $G$. Let $(\gz, \P)$ be the
space of sample paths for the random walk generated by $\mu$, and let
$w_n$ be the location of the sample path $\w$ at time $n$. Let $\Gbar$
be the relative space for $G$ union its Gromov boundary.  Define a
sequence of measures on $\Gbar \cross \Gbar$ by $X_n(U) = \P((w_n,
w_n^{-1}) \in U)$. Then
\[ \lim_{n \to \infty} X_n = \nu \cross \rnu, \]
i.e. the measures $X_n$ converge in distribution to $\nu \cross
\rnu$, where $\nu$ is the harmonic measure determined by $\mu$, and
$\nu$ is the reflected harmonic measure determined by $\mu$.
\end{theorem}

We will make use of the following results from \cite{maher1} and
\cite{maher2}. The first says that a random walk gives rise to an
element of the mapping class group which is not conjugate to a
relatively short element with asymptotic probability one.

\begin{theorem} \cite{maher1}
\label{theorem:pa}
Let $G$ be the mapping class group of an orientable surface of finite
type, which is not a sphere with three or fewer punctures. Let $\mu$
be a probability distribution on $G$ whose support generates a
non-elementary subgroup of $G$.  Then for any constant $R$, the
probability that $w_n$ is conjugate to an element of relative length
at most $R$, tends to zero as $n$ tends to infinity.
\end{theorem}

The second states that both the harmonic and the $\mun{n}$-measures of
a halfspace $\overline{H(1,x)}$ decay exponentially in $\nhat{x}$.

\begin{lemma} \cite{maher2}
\label{lemma:exponential} \label{lemma:mu-n-decay}
Let $\mu$ be a finitely supported probability distribution on $G$
whose support generates a non-elementary subgroup, and let $\nu$ be
the corresponding harmonic measure. Then there are constants $\Kexp,
Q$ and $L<1$, such that $ \nu(\overline{H(1,x)}) \leqslant
L^{\nhat{x}}$, and $ \mun{n}(H(1,x)) \leqslant Q L^{\nhat{x}}$, for
$\nhat{x} \geqslant \Kexp$.  The constants $\Kexp, Q$ and $L$ depend
on $\mu$ and $\delta$, but not on $x$ or $n$.
\end{lemma}

We now prove Theorem \ref{theorem:independence}.

\begin{proof} (of Theorem \ref{theorem:independence}) 
We start by showing that it
suffices to show that $X_{2n}$ converges to $\nu \cross \rnu$. 
Let $r$ be the diameter of the support of $\mu$, and let $A$ be an
open set in $\Gbar$. We shall write $\eta_+(A)$ for an
$r$-neighbourhood of $A$ in the relative metric, and $\eta_-(A)$ for
$A \setminus \eta_r(\Gbar \setminus A)$. In particular, the limit sets
of $A$, $\eta_+(A)$ and $\eta_-(A)$ in the Gromov boundary $\d G$ are
all the same. The distance between two positions of a sample path at
adjacent times is at most $r$, so if a sample path of length $2n$ lies
in $\eta_-(A)$, then next location of the sample at time $2n+1$ lies
in $A$, and similarly the location of the sample path at time $2n+2$ lies in
$\eta_+(A)$. Therefore if $A$ and $B$ are open sets in $\Gbar$, then 
\[ X_{2n}(\eta_-(A) \cross \eta_-(B)) \le X_{2n+1}(A \cross B) \le
X_{2n+2}(\eta_+(A) \cross \eta_+(B)). \]
Therefore, if $X_{2n}$ converges to $\nu \cross \rnu$, then both the
right and left hand terms in the inequality above converge to
$\nu(\overline{A}) \cross \rnu(\overline{B})$, and so the central term
also converges to $\nu(\overline{A}) \cross \rnu(\overline{B})$.

We now show that $X_{2n}$ converges to $\nu \cross \rnu$. At time $n$,
the location of a sample path $w_n$ determines a halfspace
$H(1,w_n)$. It is most likely that the sample path will still be in
this halfspace after another $n$ steps, as illustrated below in Figure
\ref{picture1}. We now make this precise.

\begin{figure}[H] 
\begin{center}
\epsfig{file=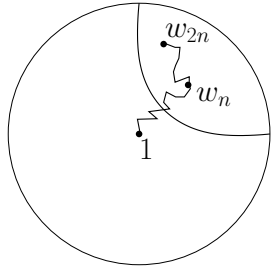, height=100pt}
\end{center}
\caption{Staying in a halfspace.}\label{picture1}
\end{figure}

\begin{claim} 
\label{claim:walk in halfspace} 
The probability that $w_{2n}$ lies in $H(1,w_n)$ tends to one as $n$
tends to infinity.
\end{claim}

\begin{proof}
If $w_n = g$, then the probability that after another $n$ steps
the sample path lies in $H(1,g)$ is $\mun{n}(g^{-1}H(1,g)) =
\mun{n}(H(g^{-1},1))$. Summing over all $g \in G$ gives
\begin{align}
\P(w_{2n} \in H(1,w_n)) & = \sum_{g \in G} \mun{n}(g)
\mun{n}(H(g^{-1},1)). \label{eq:halfspace}
\intertext{%
The relative space $\Grel$ is the union of $H(1, g^{-1})$ and
$H(g^{-1}, 1)$, so $\mun{n}(H(g^{-1},1)) \ge 1 - \mun{n}( H(1,g^{-1}) )$,
which implies%
}
\P(w_{2n} \in H(1,w_n)) & \ge \sum_{g \in G} \mun{n}(g) ( 1 -
\mun{n}(H(1, g^{-1})) ). \notag
\intertext{%
As all terms are positive, and we wish to find a lower bound, we may
discard those terms corresponding to $g \in \Bhat_r$, where
$\Bhat_r$ is the ball of radius $r$ about the identity in the
relative space.% 
}
\P(w_{2n} \in H(1,w_n)) & \ge \sum_{g \in G \setminus \Bhat_r} \mun{n}(g) ( 1 - \mun{n}(H(1,
g^{-1})) ) \notag
\intertext{%
For any $r \ge \Kexp$, we may apply the exponential decay estimate for
$\mun{n}$, Lemma \ref{lemma:mu-n-decay}.%
}
\P(w_{2n} \in H(1,w_n)) & \ge \mun{n}(G \setminus \Bhat_r) ( 1 - QL^r
) \notag
\intertext{%
The convolution measures $\mun{n}$ converge to $\nu$, and $\nu(G
\setminus \Bhat_r) = 1$, which implies%
}
\lim_{n \to \infty} \P(w_{2n} \in H(1,w_n)) & \ge 1 - QL^r \notag
\end{align}
for all $r \ge \Kexp$, and so the claim follows as $L < 1$.
\end{proof}

We now observe that we can make a similar statement, using the
reflected random walk.

\begin{claim} \label{claim:rwalk in halfspace}
The probability that $w_{2n}^{-1}$ lies in $H(1, w_{2n}^{-1} w_n )$ tends to
one as $n$ tends to infinity.
\end{claim}

\begin{proof}
This follows immediately from the argument above, as if
$w_{2n}^{-1} w_n = g$, then the probability that
$w_{2n}^{-1}$ lies in $H(1,g)$ is $\rmun{n}(g^{-1}H(1,g)) =
\rmun{n}(H(g^{-1},1))$. Again, summing over $g$ we obtain
\[ \P(w_{2n}^{-1} \in H(1,w_{2n}^{-1} w_n ))
= \sum_{g \in G} \rmun{n}(g) \rmun{n}(H(g^{-1},1)) \]
which is precisely the same as \eqref{eq:halfspace}, except with the reflected
measure $\rmu$ instead of $\mu$.
\end{proof}

The two events described in Claims \ref{claim:walk in halfspace} and
\ref{claim:rwalk in halfspace} need not be independent, but as they
both happen with probabilities that tend to one, the probability that
they both occur tends to one.  Furthermore, for any subset of $\gz$
with non-zero measure, the probability that both these events occur in
this subset tends to one as $n$ tends to infinity.

We now find upper and lower bounds for $X_{2n}$ in terms of a
distribution $Y_n$, which we now define.  Consider the sequence of
random variables $w \mapsto (w_n, w_{2m}w_n^{-1})$. We shall write
$Y_n$ for the corresponding measures on $\Gbar \cross \Gbar$, defined
by $Y_n(U) = \P( (w_n,w_{2n}^{-1}w_n) \in U )$. If we write this in
terms of the increments of the random walk, then $(w_n, w_{2n}^{-1}
w_n)$ is given by $(a_1 \ldots a_n, a_{2n}^{-1} \ldots a_{n+1}^{-1})$.
Here $a_i$ is the $i$-th increment of the random walk, so the $a_i$
are distributed as independent random variables with distribution
$\mu$.  Therefore, the distribution of $Y_n$ in each factor is
independent, so the distribution $Y_n$ is equal to $\mun{n} \cross
\rmun{n}$ on $\Gbar \cross \Gbar$.  Therefore the sequence $Y_n$
converges in distribution to $\nu \cross \rnu$.

Let $T_1$ and $T_2$ be two closed sets in the Gromov boundary
$\fmin$. The distribution $Y_n$ is equal to $\mun{n} \cross \rmun{n}$,
so $Y_n(T_1 \cross T_2) = \mun{n}(T_1) \rmun{n}(T_2)$. Recall that
$N_r(T)$ is the union of all halfspace neighbourhoods $H(1,x)$, with
$\nhat{x} \ge r$, whose closures intersect $T$.  We will write
$N^2_r(X)$ for $N_r(N_r(T))$, and so on.

As sample paths converge to the boundary almost surely, Theorem
\ref{theorem:converge}, the probability that $\nhat{w_n} \ge r$ tends
to one as $n$ tends to infinity. Therefore, the probability that all
of the following three events occur at the same time tends to one as
$n$ tends to infinity.
\begin{align*}
& \nhat{w_n} \ge r \\
& w_{2n} \in H(1,w_n) \tag{Claim \ref{claim:walk in halfspace}} \\
& w_{2n}^{-1} \in H(1,w_{2n}^{-1} w_n ) \tag{Claim \ref{claim:rwalk in halfspace}}
\end{align*}

If the first two events above occur, and $w_n \in N_r(T_1)$, then
$w_{2n} \in N^2_r(T_1)$.  Similarly, if the first and last events
above occur and $w_{2n}^{-1} w_n(\w) \in N_r(T_2)$, then $w_{2n}^{-1}
\in N^2_r(T_2)$. Therefore
\begin{align*}
\lim_{n \to \infty} Y_n(N_r(T_1) \cross N_r(T_2)) \le &
\lim_{n \to \infty} X_{2n}(N^2_r(T_1) \cross N^2_r(T_2)).
\intertext{%
Furthermore, if the first two events above occur, and $w_{2n}$
lies in $N^2_r(T_1)$, then $H(1,w_n)$ intersects $N^2_r(T_1)$, so
$w_n$ lies in $N^3_r(T_1)$. Similarly, if the first and last
events above occur, and $w_{2n}^{-1}$ lies in $N^2_r(T_2)$, then
$H(1,w_{2n}^{-1} w_n)$ intersects $N^2_r(T_2)$, so $w_{2n}^{-1} w_n$
lies in $N^3_r(T_2)$. Therefore
}
& \lim_{n \to \infty} X_{2n}(N^2_r(T_1) \cross N^2_r(T_2))
\le \lim_{n \to \infty} Y_n(N^3_r(T_1) \cross N^3_r(T_2)).
\intertext{%
The distribution $Y_n$ is given by $\mun{n} \cross \rmun{n}$,
which has limit $\nu \cross \nu$, so we have shown
}
\nu(N_r(T_1))\rnu(N_r(T_2)) \le & \lim_{n \to \infty}
X_{2n}(N^2_r(T_1) \cross N^2_r(T_2)) \le
\nu(N^3_r(T_1))\rnu(N^3_r(T_2)),
\end{align*}
and this is true for all $r$.  Recall that we showed in Lemma
\ref{lemma:neighbourhood} that $N^2_r(T) \subset N_{r-\Knbd}(T)$, for
some constant $\Knbd$, which only depends on $\delta$, and furthermore
this implies that $N^3_r(T) \subset N_{r - 2\Knbd}(T)$. In Lemma
\ref{lemma:intersections} we showed that $\bigcap N_r(T) = T$, and so
this implies that $\lim_{n \to \infty} X_{2n} = \nu \cross \rnu$, as
required.
\end{proof}

We have shown that the limiting distribution of pairs $(w_n,
w_n^{-1})$ is given by $\nu \cross \rnu$. We now use this to show that
the limiting distribution of stable and unstable endpoints of \pA
elements, $(\l^+(w_n), \l^-(w_n))$, is also given by $\nu \cross
\rnu$. Consider the sequence of random variables taking values in
$\Gbar \cross \Gbar \cup \varnothing$ given by $w \mapsto (\l^+(w_n),
\l^-(w_n))$, if $w_n$ is \pA, and $w \mapsto \varnothing$ otherwise.
We shall write $\L_n$ for the corresponding sequence of measures on
$\Gbar \cross \Gbar$, given by $\L_n(U) = \P( (\l^+(w_n), \l^-(w_n) )
\in U)$. The measures $\L_n$ need not have total mass one.  However,
elements of the mapping class group which are not \pA are all
conjugate to elements of bounded relative length, and so Theorem
\ref{theorem:pa} implies that the probability that $w_n$ is \pA tends
to one as $n$ tends to infinity.  Therefore the limiting distribution
$\lim_{n \to \infty} \L_n$ will have total mass one, and so will be a
probability measure on $\Gbar \cross \Gbar$.

\begin{lemma}
The induced measures $\L_n$ converge in distribution to the product
of the harmonic measure and the reflected harmonic measure, i.e.
\[ \lim_{n \to \infty} \L_n = \nu \cross \rnu. \]
\end{lemma}

\begin{proof}
Let $T_1$ and $T_2$ be closed sets in $\fmin$. It suffices to show
that $\lim_{n \to \infty}\L_n(T_1 \cross T_2) = \nu(T_1)\rnu(T_2)$.
We shall define the \emph{(relative) translation length}, denoted
$\t_g$, of an element $g$ of the mapping class group, to be the
translation length of $g$ acting on the relative space $\Grel$. Recall
that by \pref{fixed point}, there is a constant $\Kpa$, which only
depends on $\delta$ such that if $\t_g \ge \Kpa$, then the stable
fixed point $\l^+(g)$ is contained in $\Hbar{1, g}$.  This implies
that if $\dhat{1, w_n} \ge r$, and $\t_{w_n} \ge \Kpa$, and also $w_n$
lies in $N_r(T_1)$ and $w_n^{-1}$ lies in $N_r(T_2)$, then $\l^+(w_n)$
lies in $N^2_r(T_1)$ and $\l^-(w_n)$ lies in $N^2_r(T_2)$.  As random
walks converge to the boundary, by Theorem \ref{theorem:converge}, the
probability that $\nhat{w_n} \ge r$ tends to one as $n$ tends to
infinity, and by Theorem \ref{theorem:pa}, the probability that the
relative translation length of $w_n$ is greater than $\Kpa$ tends to
one as $n$ tends to infinity. Therefore this implies that for $r \ge \Kpa$,
\begin{align*}
\lim_{n \to \infty} X_n(N_r(T_1) \cross N_r(T_2) ) & \le \lim_{n \to
  \infty} \L_n(N^2_r(T_1) \cross N^2_r(T_2)).
\intertext{%
Similarly, if $\dhat{1, w_n} \ge r$ and $\t_{w_n} \ge \Kpa$, then if
$\l^+(w_n)$ lies in $N^2_r(T_1)$, and $\l^-(w_n)$ lies in
$N^2_r(T_2)$, then $w_n$ lies in $N^3_r(T_1)$, and $w_n^{-1}(\w)$ lies
in $N^3_r(T_2)$. As the probabilities that $\dhat{1, w_n} \ge r$ and
$\t_{w_n} \ge \Kpa$ both tend to one as $n$ tends to infinity, this
implies that for $r \ge \Kpa$,}
\lim_{n \to \infty} X_n(N_r(T_1) \cross N_r(T_2) ) & \le \lim_{n \to
  \infty} \L_n(N^2_r(T_1) \cross N^2_r(T_2)) \le \lim_{n \to \infty}
X_n(N^3_r(T_1) \cross N^3_r(T_2) ).
\end{align*}
Since $X_n$ limits to $\nu \cross \rnu$, using exactly the same
argument as in the previous proof, it follows that $\L_n$ limits to
$\nu \cross \rnu$.
\end{proof}

%%%%%%%%%%%%%%%%%%%%%%%%%%%%%%%%%%%%%%%%%%%%%%%%%%%%%%%%%%%%%%%%%%%%%%%%%%%%%%%
\section{Distances of random Heegaard splittings} \label{section:distance}
%%%%%%%%%%%%%%%%%%%%%%%%%%%%%%%%%%%%%%%%%%%%%%%%%%%%%%%%%%%%%%%%%%%%%%%%%%%%%%%

A random Heegaard splitting $M(w_n)$ is a $3$-manifold constructed
from a Heegaard splitting using a gluing map $w_n$ which is the
position of a random walk of length $n$ on the mapping class group.
Recall that the distance of a Heegaard splitting is the minimal
distance in the complex of curves between the disc sets for each
handlebody. If the disc set for one handlebody is $\ds$, then the disc
set for the other handlebody is $w_n\ds$, so the splitting distance is
the minimum distance between these two translates of the disc set.
The complex of curves is quasi-isometric to the relative space, so it
suffices to show that the distance in the relative space between the
images of the discs sets in the relative space grows linearly.

Before we begin, we record the elementary observation that as $\nu$
and $\rnu$ are non-atomic, the diagonal in $\fmin \cross \fmin$ has
measure zero in the product measure $\nu \cross \rnu$.

\begin{proposition}\label{prop:diagonal}
Let $\nu$ and $\rnu$ be non-atomic measures on $\fmin \cross
\fmin$. Then $(\nu \cross \rnu)(\D) = 0$, where $\D$ is the diagonal
$\{ (\l, \l) \mid \l \in \fmin \}$.
\end{proposition}

\begin{proof}
The measure $\nu$ is non-atomic, so for any $\e > 0$, we may partition
$\fmin$ into sets $T_i$ such that $\nu(T_i) \le \e$, for all $i$. The
diagonal in $\fmin \cross \fmin$ is covered by the union of $T_i
\cross T_i$. As
\[ (\nu \cross \rnu)(T_i \cross T_i)  = \nu(T_i)\rnu(T_i) \le \e \rnu(T_i), \]
this implies that $(\nu \cross \rnu)(\D) \le \e$, for all $\e > 0$.
\end{proof}

We will make use of the fact that a random walk makes linear progress
in the relative space.

\begin{theorem}\cite{maher2} \label{theorem:linear}
Let $G$ be the mapping class group of an orientable surface of finite
type, which is not a sphere with three or fewer punctures, and
consider the random walk generated by a finitely supported probability
distribution $\mu$, whose support generates a non-elementary subgroup
of the mapping class group. Then there is a constant $\ell > 0$ such
that $\lim_{n \to \infty} \tfrac{1}{n}\nhat{w_n} = \ell$ almost
surely.
\end{theorem}

We use the result above, together with the estimate for $\t_{w_n}$ in
terms of $\dhat{1, w_n}$ and the distance to a quasi-axis for $w_n$,
from \pref{translation}, to show that the relative translation length
$\t_{w_n}$ of $w_n$ grows linearly in $n$.

\begin{lemma} \label{lemma:linear translation length}
Let $G$ be the mapping class group of an orientable surface of finite
type, which is not a sphere with three or fewer punctures. Consider a
random walk generated by a finitely supported probability distribution
$\mu$ on $G$, whose support generates a non-elementary subgroup.  Then
there is a constant $\ell > 0$ such that
\[ \P( \norm{\tfrac{1}{n}\t_{w_n} - \ell} \le \e ) \to 1 \text{ as }
n \to \infty, \] 
for all $\e > 0$, where $\t_{w_n}$ is the translation length of the
group element $w_n$ acting on the relative space.
\end{lemma}

\begin{proof}
Given a pair of halfspaces $P = \{ H(1, x_1), H(1, x_2) \}$, with disjoint
closures, define $D(P)$ to be the supremum of the distance from $1$ to
any $\Ksplit$-\qg with one endpoint in $\Hbar{1, x_1}$ and the other
endpoint in $\Hbar{1, x_2}$, where $\Ksplit$ is the constant from
Theorem \ref{theorem:uniform}.  Define $\U_r$ to be the collection of all
pairs of halfspaces $P$, such that the two halfspaces have disjoint
closures, and $D(P) \le r$, and let 
\[U_r = \bigcup_{P \in \U_r} \Hbar{1, x_1} \cross \Hbar{1, x_2}. \]
The probability that $w_n$ is \pA, with $(\l^+(w_n), \l^-(w_n)) \in
U_r$ is $\L_n( U_r )$.  Recall that by \pref{translation}, if
$\t_{w_n} \ge \Ktrans$ then $\t_{w_n} \ge \dhat{1, w_n} - 2 \dhat{1,
  \a_{w_n}} - \Ktrans$, where $\a_{w_n}$ is a $\Ksplit$-quasi-axis for
$w_n$, and $\Ktrans$ depends only on $\Ksplit$ and $\delta$.  The set
of elements of translation length at most $\Ktrans$ forms a set of
elements all of which are conjugate to elements of bounded relative
length, so by Theorem \ref{theorem:pa}, the probability that $\t_{w_n}
\ge \Ktrans$ tends to one as $n$ tends to infinity.  If $(\l^+(w_n),
\l^-(w_n)) \in U_r$, then $\dhat{1, \a_{w_n}} \le r$, by the
definition of $U_r$, so the probability that $(\l^+(w_n), \l^-(w_n))
\in U_r$ and 
\begin{equation} \label{eq:translation bound} 
\dhat{1, w_n} \ge \t_{w_n} \ge \dhat{1, w_n} - 2r - \Ktrans
\end{equation} 
tends to $(\nu \cross \rnu)(U_r)$ as $n$ tends to infinity, where the
left hand inequality follows from the fact that $\dhat{1, g}$ is an
upper bound for $\t_g$, for any element $g$.

The random walk makes linear progress in the relative space, Theorem
\ref{theorem:linear}, so there is a constant $\ell > 0$ such that
$\tfrac{1}{n}\nhat{w_n} \to \ell$ as $n$ tends to infinity almost
surely.  Therefore, using \eqref{eq:translation bound}, for any $\e >
0$, the limiting probability that $w_n$ has relative translation
length between $(\ell + \e)n$ and $(\ell - \e)n$ is at least $(\nu
\cross \rnu)(U_r)$, as $n$ tends to infinity, i.e.
\[ \lim_{n \to \infty} \P( \norm{\tfrac{1}{n}\t_{w_n} - \ell} \le
\e ) \ge (\nu \cross \rnu)(U_r), \]
for all $r$.

Therefore, in order to complete the proof of Lemma \ref{lemma:linear
  translation length}, it suffices to show that $\liminf U_r = \fmin
\cross \fmin \setminus \D$, where 
\[ \liminf U_r = \bigcup_s \bigcap_{r \ge s} U_r, \]
and $\D$ is the diagonal in $\fmin \cross \fmin$. This is because
$(\nu \cross \rnu)(\D) = 0$, by \pref{diagonal}, and $(\nu \cross
\rnu)(\liminf U_r) = 1$ implies that $(\nu \cross \rnu)(U_r)$ tends to
$1$ as $r$ tends to infinity.

We now complete the proof of Lemma \ref{lemma:linear translation
  length} by showing that $\liminf U_r = \fmin \cross \fmin \setminus
\D$.  Given distinct points $\l_1$ and $\l_2$ in $\fmin$, choose a
$\Ksplit$-\qg $\a$ with endpoints $\l_1$ and $\l_2$, and let $p$ be a
closest point on $\a$ to $1$. Recall that by \pref{disjoint}, there
are constants $A$ and $B$, which only depend on $\Ksplit$ and
$\delta$, such that for any points $x_1$ and $x_2$ on $\a$ with
$\dhat{p, x_i} \ge 2 \dhat{1, p} + A$, the halfspaces $\Hbar{1, x_1}$
and $\Hbar{1, x_2}$ are disjoint, and $\dhat{1, \b} \le \dhat{1, \a} +
B$ for any $\Ksplit$-\qg $\b$ with one endpoint in each of $\Hbar{1,
  x_1}$ and $\Hbar{1, x_2}$.  Therefore, $\Hbar{1, x_1} \cross
\Hbar{1, x_2} \subset U_r$, for all $r \ge \dhat{1, \a} + B$. As
$(\l_1, \l_2) \in \Hbar{1, x_1} \cross \Hbar{1, x_2}$, this implies
$(\l_1, \l_2) \in \liminf U_r$, for all $\l_1 \not = \l_2$. So
$\liminf U_r = \fmin \cross \fmin \setminus \D$, as required.
\end{proof}

We now show that the distance of a random Heegaard splitting grows
linearly. We will use the fact that the disc set $\ds$ is quasiconvex,
which was shown by Masur and Minsky \cite{mm3}.

\begin{theorem} \cite{mm3} \label{theorem:quasiconvex}
There is a constant $\Kqc$, which only depends on the genus of the
handlebody, such that the disc set of a handlebody is a
$\Kqc$-quasiconvex subset of the complex of curves.
\end{theorem}

We shall define the disc set in the relative space to be the image of
the disc set in the complex of curves under the quasi-isometry from
the complex of curves to the relative space. Then the disc set in the
relative space is also quasiconvex, with quasiconvexity constant at
most $Q \Kqc$, where $Q$ is the quasi-isometry constant between the
relative space and the complex of curves. We will abuse notation by
also referring to the disc set in the relative space as $\ds$.

We will also need to know that the harmonic measure of the disc set is
zero, Theorem \ref{theorem:disc set}, so we will need to assume that
the support of the probability distribution $\mu$ generates a
semi-group containing a complete subgroup of the mapping class group.
We now show that for any set $X$ which is quasiconvex and whose limit
set has harmonic measure zero, the distance between $X$ and $w_nX$
grows linearly in $n$. This implies Theorem \ref{theorem:main}, where
the constants $\ell_1$ and $\ell_2$ may be chosen to be $\ell/Q$ and
$\ell Q$.

\begin{theorem} \label{theorem:last}
Let $G$ be the mapping class group of a closed orientable surface.
Consider a random walk generated by a finitely supported probability
distribution $\mu$ on $G$, and whose support generates a semi-group
containing a complete subgroup of the mapping class group.  Let $X$ be
a quasiconvex subset of the relative space, whose limit set has
measure zero with respect to both harmonic measure and reflected
harmonic measure. Then there is a constant $\ell > 0$ such that
\[ \P( \norm{\tfrac{1}{n}\dhat{X, w_n X} - \ell} \le \e ) \to 1 \text{ as }
n \to \infty, \] 
for all $\e > 0$, where $\dhat{X, w_n X}$ is the minimum distance between
$X$ and $w_n X$.
\end{theorem}

\begin{proof}
Let $H(1, x_1)$ be a halfspace which is $K$-nested inside $H(1, y_1)$,
and let $H(1, x_2)$ be a halfspace which is $K$-nested inside $H(1,
y_2)$.  We will call a pair of such nested halfspaces a
\emph{$K$-nested pair}. We will refer to $H(1, x_1)$ and $H(1, x_2)$
as the \emph{inner pair} of halfspaces, and $H(1, y_1)$ and $H(1,
y_2)$ as the \emph{outer pair}.  We say a $K$-nested pair is
\emph{$X$-disjoint}, if the closures of the outer pair, $\Hbar{1,
  y_1}$ and $\Hbar{1, y_2}$ are disjoint, and are also disjoint from
$X$.  We shall choose the constant $K$ to be the constant $B$ from
\pref{disjoint}, which only depends on $\delta$, and a choice of \qg
constant, which we will choose to be the constant $\Ksplit$ from
Theorem \ref{theorem:uniform}.

Given a $K$-nested pair $P = \{ H(1, x_1) \subset H(1, y_1), H(1, x_2)
\subset H(1, y_2) \}$, define $D(P)$, to be the maximum distance from
$1$ to any $\Ksplit$-\qg with one endpoint in each of the closures of
the inner halfspaces, $\overline{H(1,x_1)}$ and $\overline{H(1,x_2)}$.
Let $R(P)$ be the diameter of the \npp of the closure of the
complement of $H(1, y_1) \cup H(1, y_2)$ to any $\Ksplit$-\qg with one
endpoint in $\overline{H(1,x_1)}$ and the other in $\Hbar{1,x_2}$.
Given a number $r$, define $\U_r$ to be the collection of all pairs of
$K$-nested, $X$-disjoint halfspaces $P$ such that $D(P) \le r$ and
$R(P) \le r$, and let 
\[ U_r = \bigcup_{P \in \U_r} \Hbar{1, x_1} \cross \Hbar{1, x_2}, \]
where $H(1, x_1)$ and $H(1, x_2)$ are the inner pair of $P$.

Consider a \pA element $g$, with $\Ksplit$-quasi-axis $\a$, with
endpoints $(\l^+(g),\l^-(g)) \in \Hbar{1, x_1} \cross \Hbar{1, x_2}$,
where the two halfspaces are the inner pair of some pair $P \in \U_r$.
As $X$ is contained in the complement of the outer pair of halfspaces
of $P$, the nearest point projection of $X$ to $\a$ has diameter at
most $r$, by the definition of $U_r$. As $g$ moves every point
distance at least $\t_g$, the distance between the \npps of $X$ and
$gX$ to $\a$ is at least $\t_g - r$.  Segments of $\a$ lie in
$\Kn$-neighbourhoods of geodesics $[\a_{-n}, \a_{n}]$ between their
endpoints, where $\Kn$ is a geodesic neighbourhood constant for $\a$,
which only depends on $\Ksplit$.  By \pref{close}, the closest point
projection to a segment of $\a$ is distance at most $3\Kn + 6 \delta$
from its \npp to the geodesic $[\a_{-n}, \a_n]$.  We may estimate the
distance between $X$ and $gX$ in terms of the distance between their
\npps to $[\a_n, \a_{-n}]$, as by \pref{double}, $\dhat{X, gX} \ge
\dhat{\pi(X), \pi(gX)} - 24 \delta$, where $\pi$ is the \npp map onto
$[\a_n, \a_{-n}]$. This implies that there is a constant $L = 6 \Kn +
36 \delta$, which only depends on $\Ksplit$ and $\delta$, such that
$\dhat{X, gX} \ge \t_g - r - L$.

The probability that $w_n$ is \pA, with stable and unstable foliations
$(\l^+(w_n), \l^-(w_n)) \in U_r$, is $\L_n(U_r)$. Furthermore the
translation length $\t_{w_n}$ grows linearly, Theorem
\ref{theorem:linear}, and the measures $\L_n$ converge to $\nu \cross
\rnu$, by Theorem \ref{theorem:independence}. Therefore for any $r$,
and any $\e> 0$,
\[ \lim_{n \to \infty} \P ( \norm{ \tfrac{1}{n} \dhat{X, w_n X} -
  \ell } \le \e ) \ge (\nu \cross \rnu )(U_r),  \]
where $\ell$ is the constant from Lemma \ref{lemma:linear translation
  length}.  Therefore, in order to complete the proof of Theorem
\ref{theorem:last}, it suffices to show that 
\begin{equation} \label{eq:liminf} 
\liminf U_r = (\fmin \setminus X) \cross (\fmin \setminus X)
\setminus \D. 
\end{equation}
This is because the diagonal has $(\nu \cross \rnu)$-measure zero, by
\pref{diagonal}, and we have assumed that $X$ has measure zero with
respect to both $\nu$ and $\rnu$. Therefore, $(\nu \cross
\rnu)(\liminf U_r) = 1$, which implies that $(\nu \cross \rnu)(U_r)$
tends to one as $n$ tends to infinity.

We now complete the proof of Theorem \ref{theorem:last} by verifying
line \eqref{eq:liminf}.  Given any pair of foliations $(\l_1, \l_2)
\in (\fmin \setminus X) \cross (\fmin \setminus X)$, with $\l_1 \not =
\l_2$, we may choose a $\Ksplit$-\qg $\a$ with $\l_1$ as the positive
endpoint and with $\l_2$ as the negative endpoint, i.e. $\l_1 =
\lim_{n \to \infty} \a_n$, and $\l_2 = \lim_{n \to \infty} \a_{-n}$.
Let $p$ be a closest point on $\a$ to $1$.  Recall that by
\pref{disjoint}, there are constants $A, B$ and $C$, which only depend
on $\Ksplit$ and $\delta$, such that for any points $x_1$ and $x_2$ on
$\a$, separated by $p$, with $\dhat{p, x_i} \ge 2 \dhat{1, p} + A$,
there is a pair $P$ of $B$-nested halfspaces $H(1, x_1) \subset H(1,
y_1)$ and $H(1, x_2) \subset H(1, y_2)$, such that $\Hbar{1, y_1}$ and
$\Hbar{1, y_2}$ are disjoint, and any $\Ksplit$-\qg $\b$ with one
endpoint in each of $\Hbar{1, x_1}$ and $\Hbar{1, x_2}$ is contained
in the union of the outer halfspaces together with an
$A$-neighbourhood of $\a$. In particular this implies that $\dhat{1,
  \b} \le \dhat{1, \a} + A$, and so $D(P) \le \dhat{1, \a} + A$.
Furthermore, the nearest point projection of the complement of $H(1,
y_1) \cup H(1, y_2)$ to $\b$ has diameter at most $\half \dhat{x_1,
  x_2} + C$, which implies that $R(P) \le \half \dhat{x_1, x_2} + C$.
As $\l_1$ and $\l_2$ are disjoint from $X$, which is closed, we may
find a pair $P$ of $B$-nested halfspaces, $H(1, x_1) \subset H(1,
y_1)$ and $H(1, x_1) \subset H(1, y_2)$, with the outer halfspaces
$\Hbar{1, y_1}$ and $\Hbar{1, y_2}$ disjoint from $X$, and with $D(P)
\le \dhat{1, \a} + A$, and $R(P) \le \half \dhat{x_1, x_2} +C$.
Therefore, $(\l_1, \l_2) \in U_r$, for all $r \ge \dhat{1, \a} + A +
\half \dhat{x_1, x_2} + C$, and so $\liminf U_r = (\fmin \setminus X)
\cross (\fmin \setminus X) \setminus \D$.
\end{proof}

Finally, we remark that this also gives the result stated in the
introduction regarding random Heegaard homology spheres obtained by a
nearest neighbour random walk on the Torelli group. In fact, the union
of the disc set $\ds$ with any fixed translate $g\ds$ is quasiconvex,
though possibly with a different quasiconvexity constant, and also has
limit set having harmonic measure zero. This implies that the distance
between $g\ds$ and $w_n\ds$ grows linearly in $n$, so by choosing a
random walk supported on the Torelli group, the homology of $M(w_n g)$
is the same as that of $M(g)$, so we can create random Heegaard
splittings which are homology spheres, or which have the homology of
any fixed $3$-manifold.

%%%%%%%%%%%%%%%%%%%%%%%%%%%%%%%%%%%%%%%%%%%%%%%%%%%%%%%%%%%%%%%%%%%%%%%%%%%%%%%

%%%%%%%%%%%%%%%%%%%%%%%%%%%%%%%%%%%%%%%%%%%%%%%%%%%%%%%%%%%%%%%%%%%%%%%%%%%%%%%

\end{document}